\documentclass[12pt]{article}
\usepackage{makeidx}
\usepackage{amsfonts}
\usepackage{amsmath}
\usepackage{amssymb}
\usepackage{graphicx}
\usepackage{amsopn}
\usepackage{amsthm}
\usepackage[bookmarks,colorlinks=true]{hyperref}
\newtheorem{theorem}{Theorem}[section]

\newtheorem{corollary}[theorem]{Corollary}
\newtheorem{lemma}[theorem]{Lemma}
\newtheorem{notation}[theorem]{Notation}
\newtheorem{proposition}[theorem]{Proposition}
\theoremstyle{definition}
\newtheorem{definition}[theorem]{Definition}

\newtheorem{question}[theorem]{Question}
\newtheorem{remark}[theorem]{Remark}

\begin{document}

\title{Lattice initial segments of the hyperdegrees}
\author{Richard A. Shore\thanks{%
Partially supported by NSF Grant DMS-0554855.} \\
%EndAName
Department of Mathematics\\
Cornell University\\
Ithaca NY 14853 \and Bj\o rn Kjos-Hanssen\thanks{%
Partially supported as co-PI by NSF grant DMS-0652669. \newline
The authors also thank the referee for suggestions that improved the
presentation in several ways.} \\
%EndAName
Department of Mathematics\\
University of Hawai{\textquoteleft}i at M{\=a}noa\\
Honolulu HI 96822}
\maketitle

\begin{abstract}
We affirm a conjecture of Sacks [1972] by showing that every countable
distributive lattice is isomorphic to an initial segment of the
hyperdegrees, $\mathcal{D}_{h}$. In fact, we prove that every sublattice of
any hyperarithmetic lattice (and so, in particular, every countable locally
finite lattice) is isomorphic to an initial segment of $\mathcal{D}_{h}$.
Corollaries include the decidability of the two quantifier theory of $%
\mathcal{D}_{h}$ and the undecidability of its three quantifier theory. The
key tool in the proof is a new lattice representation theorem that provides
a notion of forcing for which we can prove a version of the fusion lemma in
the hyperarithmetic setting and so the preservation of $\omega _{1}^{CK}$.
Somewhat surprisingly, the set theoretic analog of this forcing does not
preserve $\omega _{1}$. On the other hand, we construct countable lattices
that are not isomorphic to an initial segment of $\mathcal{D}_{h}$.
\end{abstract}
\newpage
\tableofcontents

\section{Introduction}

Given a notion of relative computability or complexity $\leq _{r}$ on sets $A
$ (of natural numbers) or functions $f$ (from $\mathbb{\omega }$ to $\mathbb{%
\omega }$) the corresponding degree structure $\mathcal{D}_{r}$ consists of
the equivalence classes $\deg _{r}(A)$ ($\deg _{r}(f)$) of mutually
computable sets (functions) with the ordering induced by the given
reducibility $\leq _{r}$. These classes are called the $r$-degrees. The
classic example is Turing computability, Turing reducibility $\leq _{T}$ and
the Turing degrees but many others have been studied ranging from
polynomial-time to constructibility. The first level of the scaffolding on
which our understanding of such degree structures is built provides an
analysis of the partial orders that can be embedded in them.

For the Turing degrees the story begins with Kleene and Post [1954] who
proved that every countable partial order can be embedded in $\mathcal{D}%
_{T} $. They used finite approximation methods that today would be called
Cohen forcing in arithmetic. Stronger results about embedding uncountable
partial orders were proven, for example, by Sacks [1963] but the full
question of whether every partial order of size $2^{\aleph _{0}}$ with $%
\{x|x\leq z\}$ countable for every $z$ can be embedded in $\mathcal{D}_{T}$
remains open. Nonetheless, the countable case is more than enough to decide
all one quantifier sentences of $\langle \mathcal{D}_{T},\leq _{T}\rangle $
as being true if and only if consistent with the axioms of partial orders.

The next level of analysis deals with extension of embedding questions such
as density or minimality and, more generally, questions about when a given
realization of a partial order in $\mathcal{D}$ can be extended to a
specified larger partial order. In $\mathcal{D}_{T}$, the long journey along
this road began with Spector's [1956] construction of a minimal degree. His
method can be viewed as using full binary trees in place of finite
characteristic functions for the approximations to the desired set. These
methods were greatly extended by many researchers to embed more and more
lattices as initial segments of $\mathcal{D}_{T}$. We mention a few of the
key steps: Lachlan [1968] showed that every countable distributive lattice
is isomorphic to an initial segment of $\mathcal{D}_{T}$ (and so that its
theory is undecidable); Lerman [1971] did the same for all finite lattices;
Lachlan and Lebeuf [1976] for countable uppersemilattices (\emph{usls}) and
Abraham and Shore [1986] for all usls of size $\aleph _{1}$. (Groszek and
Slaman [1983] show that it is consistent with ZFC that there are lattices of
size $\aleph _{2}\leq 2^{\aleph _{0}}$ with $\{x|x\leq z\}$ countable for
every $z$ that can not be embedded in $\mathcal{D}_{T}$ as initial
segments.) The embedding methods in the countable situations involve
approximations by more and more complicated types of trees whose shape is
tied to a series of representation theorems for (uppersemi)lattices. The
uncountable ones need, in addition, some set theoretic techniques.

At this level of our scaffolding, Lerman's [1971] result on embedding finite
lattices shows that the three quantifier theory of $\mathcal{D}_{T}$ is
undecidable (Schmerl, see Lerman [1983]). Coupled with the methods of Kleene
and Post [1954], it also suffices to establish the decidability of the two
quantifier theory of this structure (Shore [1978] and Lerman [1983, VII.4]).
The case of countable recursive lattices suffices (Miller, Nies and Shore
[2004]) to show the undecidability of the two quantifier theory in the
language extended by symbols for both join and infimum with the
understanding that the result applies to any total extension of the infimum
relation on $\mathcal{D}_{T}$ which, by Kleene and Post [1954], is not a
lattice. (This proof also supplies another one for the undecidability of the
three quantifier theory in the language with just $\leq _{T}$.)

The same results on embedding as partial orderings and initial segments, and
so the corresponding applications to the analysis of their theories, can be
proved by quite analogous types of constructions (Cohen and perfect tree
forcing) for a range of reducibilities from truth table to arithmetic, $%
A\leq _{a}B\Leftrightarrow A\leq _{T}B^{(n)}$ for some $n\in \omega $, (see
e.g. Feferman [1965], Nerode and Shore [1980], Odifreddi [1983], M. Simpson
[1985]). When one moves from the realm of recursion theory and the natural
numbers to set theory and the ordinals, new issues arise.

If we look, for example, at relative constructibility ($A\leq
_{c}B\Leftrightarrow A\in L[B]$ for $A,B\subseteq \omega $) and the
constructibility degrees, $\mathcal{D}_{c}$, of subsets of $\omega $, we are
faced with the new problem of preserving $\omega _{1}$ in our forcing
extensions. For Cohen forcing this was part of Cohen's seminal results and
one can carry over (under suitable set theoretic hypotheses which we ignore
here) the partial order embedding (Cohen [1966]) and related results to $%
\leq _{c}$ with little difficulty (see, e.g. Balcar and Hajek [1978],
Farrington [1983], [1984] and Abraham and Shore [1986a]). Our ability to
preserve $\omega _{1}$ for perfect forcing is due to Sacks [1971] and is
based on what he calls the fusion lemma. It enabled Sacks [1971] to prove
the existence of a minimal degree of constructibility.

The next major steps towards determining the initial segments of $\mathcal{D}%
_{c}$ were taken by Adamowicz who proved first [1976] that all finite
lattices can be embedded as initial segments of D$_{c}$ and then [1977] that
all countable constructible well founded usls can be so embedded. The extra
assumptions Adamowicz needed for her proof turned out to be to some extent
necessary. Abraham and Shore [1986a] showed that not every countable well
founded distributive lattice is isomorphic to an initial segment of $%
\mathcal{D}_{c}$. The crucial fact here is that the relation $\leq _{c}$ is
itself constructible and so one can get the result by coding
nonconstructible sets in distributive lattices. Even if we restrict
attention to constructible lattices, some remnants of the well foundedness
assumption remain necessary. Lubarsky [1987] proved that every countable
lattice isomorphic to an initial segment of $\mathcal{D}_{c}$ is complete.
On the positive side, the best results are those of Groszek and Shore [1988]
that every countable (dual) algebraic lattice $\mathcal{L}$ (i.e.\ $\mathcal{%
L}$ is complete and generated by its compact elements $x$ for which $\wedge
I\leq x$ implies that there is a finite $F\subseteq I$ such that $\wedge
F\leq x$, for every $I\subseteq $ $\mathcal{L}$). (There is new interesting
work, however, by Dorais [2007] on the c-degrees of subsets of $\aleph _{1}$
produced by forcing with Souslin trees instead of countable perfect trees.)

Our concern in this paper is a reducibility and degree notion lying between
that of arithmetic and constructible: the hyperarithmetic degrees $\mathcal{D%
}_{h}$. For Turing reducibility and $\mathcal{D}_{T}$ and arithmetic
reducibility and $\mathcal{D}_{a}$, the setting is first order arithmetic
and the natural numbers or $\omega $. For set theory we have relative
constructibility and $\omega _{1}$. Our setting is second order arithmetic
and $\omega _{1}^{CK}$, the first nonrecursive ordinal. It takes the place
of $\omega _{1}$, the first uncountable ordinal, as its effective analog. We
say that $A$ is hyperarithmetic in $B$, $A\leq _{h}B$, if $A\leq
_{T}B^{(\alpha )}$ for some ordinal $\alpha <\omega _{1}^{B}$, i.e. the
order type of $\alpha $ has a representative recursive in $B$. Here $%
B^{(\alpha )}$ is the $\alpha ^{th}$ iterate of the Turing jump applied to $%
B $ where effective (in $B$) unions are taken at limit levels . For another
view, note that Kleene showed (see Sacks [1990, II.1-2]) that $A\leq _{h}B$
if and only if $A$ is $\Delta _{1}^{1}(B)$.

In this setting, Feferman [1965] (see also Sacks [1990]) introduced a
ramified language for second order arithmetic and the appropriate notion of
Cohen forcing. He proved (among other results) that every countable partial
order is embeddable in $\mathcal{D}_{h}$ (even below the hyperarithmetic
degree of Kleene's $\mathcal{O}$, the complete $\Pi _{1}^{1}$ set). Some
extensions and related results using Cohen forcing are in Thomason [1967]
and [1969] and also Odifreddi [1983a]. The crucial fact needed is the
preservation of $\omega _{1}^{CK}$, i.e.\ if $G$ is Cohen generic in this
setting then $\omega _{1}^{G}=\omega _{1}^{CK}$. (This turns out to be
equivalent to $\Delta _{1}^{1}$-comprehension holding in the generic
extension.)

One can also adapt perfect forcing to this setting to construct a minimal
hyperdegree (Gandy and Sacks [1967] or with more detailed exposition and
explanation Sacks [1971] or [1990]). Again the crucial issue is the
preservation of $\omega _{1}^{CK}$ (or $\Delta _{1}^{1}$-comprehension) by a
fusion lemma (see \S \ref{fusionsec} below). Here delicate definability
issues and the theory of $\Pi _{1}^{1}$ sets play a role not seen in either
the arithmetic or set theoretic settings. In contrast to all the other
degree structures discussed so far, almost nothing more has been known about
initial segments of $\mathcal{D}_{h}$.

Thomason [1970] proved that every finite distributive lattice is isomorphic
to an initial segment of $\mathcal{D}_{h}$ and there matters stood. In his
review of this paper, Sacks [1972] writes as follows:

\textquotedblleft He raises a methodologically interesting open question: is
every countable distributive lattice isomorphic to an initial segment of
hyperdegrees? The answer (according to the reviewer)\ is probably yes, but
(as the author points out)\ the method of the paper does not suffice to
prove it.\textquotedblright

Now in 1970 Lerman's methods for handling nondistributive lattices were not
yet available but for the distributive ones there was as much available then
as now. Thus the methodological issues were not about the initial segment
constructions as used in $\mathcal{D}_{T}$ but rather about making the
analog of the fusion lemma work in more general settings. Constructing
embeddings of infinite lattices required (even in the later Turing degree
constructions) imposing more and more restrictions on the trees as the
construction progressed (often using finite approximations to the lattice to
guide them). This type of forcing condition is not amenable to the arguments
for the fusion lemma as the nature of the conditions change as more elements
of the lattice are considered and so fusing an infinite sequence of such
conditions produces a tree that is not a condition. The argument for the
constructibility degrees gets around this problem by restricting attention
to constructible lattices so that the entire lattice can be treated at once
and so uses trees of a single shape just with branchings that grow to match
the uniformly constructible approximations to the given lattice. One can
then prove the set theoretic fusion lemma to preserve $\omega _{1}$ much as
for binary trees.

This approach would work for the hyperdegrees as well but would be
restricted to (dual) algebraic hyperarithmetical lattices and so, of course,
it would not suffice to embed all the distributive lattices. Now for the
constructibility degrees some such restrictions were necessary and not all
countable distributive lattices (or even linear orders) are isomorphic to
initial segments of $\mathcal{D}_{c}$. This analogy makes it seem less
likely that Sacks's conjecture about the initial segments of $\mathcal{D}%
_{h} $ is true especially since all the available techniques seem to be
quite similar.

One thus turns to finding counterexamples. The crucial fact used both in the
first examples of distributive lattices not isomorphic to initial segments
of $\mathcal{D}_{c}$ and in the later proofs of the necessity of
completeness was that $\leq _{c}$ is itself a constructible relation. For
the Turing degrees, $\leq _{T}$ is far from recursive: it is only a $\Sigma
_{3}^{0}$ relation. The hyperdegrees lie in between (in this sense as well)
as $\leq _{h}$ is a $\Pi _{1}^{1}$ relation and so analogous to $\Sigma
_{1}^{0}$ or r.e. ones in the setting of the Turing degrees. ($A\leq _{h}B$
if and only if both $A$ and its complement are $\Pi _{1}^{1}$ in $B$.) Thus
the construction of even the basic counterexamples requires more delicacy.
We here use the methods of finitely generated successor models introduced in
Shore [1981] and used in the setting of $\mathcal{D}_{h}$ in Shore [2007],
[2008] to prove Slaman and Woodin's result that $\mathcal{D}_{h}$ is rigid
and biinterpretable with second order arithmetic by purely degree theoretic
arguments that work locally (in jump ideals) rather than just in $\mathcal{D}%
_{h}$ as a whole.

As every initial segment of $\mathcal{D}_{h}$ (or any of our degree
structures) has a least element we assume that all (upper or lower
semi)lattices have a least element $0$. As we only consider countable
lattices, we also consider only (upper or lower semi)lattices with a
greatest element, $1$, as well. We think of both $0$ and $1$ as named by
constants in the language and so are preserved under substructures and
extensions, even as upper or lower semilattices.

\begin{theorem}
\label{ctrex}Not every countable lattice is isomorphic to an initial segment
of the hyperdegrees.
\end{theorem}

\begin{proof}
Shore [2007, \S 2] and [2008, \S 2-3] present a method for taking any set $X$
and constructing a lattice $\mathcal{L}_{X}$ such that if $f$ is an
embedding of $\mathcal{L}_{X}$ into $\mathcal{D}_{h}$ then $X\leq _{h}f(0_{%
\mathcal{L}_{X}})$. Thus if we take $X$ to be, for example, Kleene's $%
\mathcal{O}$ or any nonhyperarithmetic set, any embedding $f$ of $\mathcal{L}%
_{X}$ as an initial segment of $\mathcal{D}_{h}$ would contradict the
assumption that $X$ is not hyperarithmetic.
\end{proof}

We defer a more detailed explanation of this coding method to \S \ref{quest}
where we need an elaboration for a finer result. For now it suffices to
state that the lattices constructed contain finitely many elements which
generate (using $\vee $ and $\wedge $) a sequence of incomparable elements
of type $\omega $ and other parameters that define the subsets of this
sequence corresponding to $X$ and its complement. Moreover, the recovery
procedure producing $X$ (and its complement) is positively $\Sigma _{1}^{0}$
in the partial order relation and join operator of the lattice and so $\Pi
_{1}^{1}$ (and so $\Delta _{1}^{1}$) in (by an additional trick) the bottom
(and not just top) set of any embedding. Thus these lattices have two
seemingly crucial properties that would lead to our contradiction if they
could be realized as initial segments of $\mathcal{D}_{h}$. The first is
that the lattice is not hyperarithmetic. This is analogous to the first
examples in $\mathcal{D}_{c}$ except that the coding is more delicate and
survives $\leq _{h}$ being only $\Pi _{1}^{1}$ rather than hyperarithmetic.
On the other hand, this same difference makes the direct coding methods into
distributive lattices used for $\mathcal{D}_{c}$ unavailable in $\mathcal{D}%
_{h}$. The second aspect of the argument is that the coding (of
nonhyperarithmetic information) used here seems to rely on the fact that it
is finitely generated as it uses the generators as parameters in the
decoding. This type of coding cannot be carried out in distributive lattices
as they are all \emph{locally finite lattices}: every finite subset
generates a finite sublattice. This obstacle to constructing counterexamples
then revives the possibility of verifying Sacks's conjecture for
distributive lattices.

In fact, we prove that both of these properties (nonhyperarithmetic and not
locally finite) are necessary for a countable lattice not to be isomorphic
to an initial segment of $\mathcal{D}_{h}$. Perhaps surprisingly, we provide
a common generalization of a lattice being either hyperarithmetic or locally
finite.

\begin{theorem}
\label{main}Every sublattice $\mathcal{K}$ of any hyperarithmetic lattice $%
\mathcal{L}$ is isomorphic to an initial segment of the hyperdegrees. In
fact, it can be realized as an initial segment with top a hyperdegree below
that of $\mathcal{O}\oplus \mathcal{K}$.
\end{theorem}

To see that this is indeed the desired common generalization we need to show
that every countable locally finite lattice is isomorphic to a sublattice of
a hyperarithmetic lattice.

\begin{proposition}
There is a recursive, universal, locally finite lattice, i.e.\ a recursive,
locally finite lattice into which every countable, locally finite lattice
can be embedded.
\end{proposition}

\begin{proof}
With the proper organization of the requirements, the standard Fraiss\'{e}
construction (as in Hodges [1993, 7.1.2] produces the desired lattice once
one has the amalgamation property for the class of finite lattices.
\end{proof}

\begin{lemma}
\label{amal}The class of finite lattices (with $0$ and $1$) has the
amalgamation property, i.e.\ if $\mathcal{A}$, $\mathcal{B}_{0}$ and $%
\mathcal{B}_{1}$ are finite lattices and $f_{0},f_{1}$ are embeddings of $%
\mathcal{A}$ into $\mathcal{B}_{0}$ and $\mathcal{B}_{1}$, respectively,
then there is a finite lattice $\mathcal{C}$ and embeddings $g_{0}$ and $%
g_{1}$ of $\mathcal{B}_{0}$ and $\mathcal{B}_{1}$, respectively, into $%
\mathcal{C}$ such that $g_{0}f_{0}\upharpoonright \mathcal{A}%
=g_{1}f_{1}\upharpoonright \mathcal{A}$.
\end{lemma}

\begin{proof}
This should be a known fact but we have not found a reference to the precise
form of the amalgamation property that we need. We supply a proof at the end
of \S \ref{repsec}.
\end{proof}

As we mentioned above, every distributive lattice is locally finite (this is
well known and follows easily from the Stone Representation Theorem (see
e.g. Gr\"{a}tzer [2003, p. 85]) that it is isomorphic to a ring of sets).
Thus we have our answer to the original question of Thomason [1970] and
Sacks [1972].

\begin{corollary}
Every locally finite and so, in particular, every countable distributive
lattice is isomorphic to an initial segment of the hyperdegrees.
\end{corollary}

These results also allow us to establish the same fine line between
decidability and undecidability in the fragments of the theory of $\mathcal{D%
}_{h}$ as one has for $\mathcal{D}_{T}$.

\begin{theorem}
The two quantifier theory of $\mathcal{D}_{h}$ is decidable while the three
quantifier theory is undecidable as is the two quantifier theory in the
language with both $\vee $ and $\wedge $ where $\wedge $ denotes any total
extension of the infimum relation.
\end{theorem}

\begin{proof}
The proofs for the first two assertions in $\mathcal{D}_{T}$ (see Lerman
[1983 VII.4, II.4.11 and A.2.9] although some minor corrections are needed)
rely only on two facts. First, the structure is an usl with $0$ in which all
finite lattices can be embedded as initial segments. Second, it requires a
Kleene-Post type extension of embedding theorem: Given any finite usl $P$
extended as a partial order by a finite $Q$ and any embedding of $P$ into $%
\mathcal{D}_{T}$ there is an extension of the embedding to one of $Q$ as
long as $Q$ respects the usl structure of $P$ and has no new elements below
any of those in $P$. Our results supply the first fact. Cohen forcing in the
hyperarithmetic setting supplies the second. (The required extension of the
given embedding can easily be constructed from a set of hyperdegrees
independent over all the given degrees in the image of $P$ as in the
standard construction of an embedding of an arbitrary partial order. Such
independent degrees are given by the columns of a Cohen generic set (in the
hyperarithmetic sense) by relativizing Feferman [1965] to the top given
degree.) The final undecidability result needs only that all recursive
lattices are isomorphic to initial segments as in Miller, Nies and Shore
[2003, Cor. 3.2].
\end{proof}

The crucial aspect of our version of perfect trees (and associated forcing
notion) that allows us to prove the required fusion lemma is a type of
forgetfulness property (Proposition \ref{refine}). We will use pairs
consisting of trees $T$ and a finite sublowersemilattice (\emph{slsl}) $%
\mathcal{\hat{K}}$ of $\mathcal{L}$. The shape of the tree is determined in
advance by a type of sequential usl representation for $\mathcal{L}$ and the
second component indicates the congruence relations that must be respected
in all further refinements. (All these notions are made precise in \S \ref%
{uslrep}.) What goes wrong in a naive attempt to transfer the methods of
either $\mathcal{D}_{T}$ or $\mathcal{D}_{c}$ is the variation in the
congruences that must be respected. Our forcing will have the property that
if some condition forces a sentence of a certain form then the second
component is irrelevant and so can be made the same in all the extensions
needed in the fusion lemma. While this may seem unlikely, it is possible
once one has proven a new lattice representation theorem (Theorem \ref%
{repthm}). Its construction takes into account various requirements for all
sublattices at every step. The full lattice structure is exploited here by
using a slsl decomposition of $\mathcal{L}$ and usl representations for the
slsls. (The slsls are actually lattices as they are finite although with
elements having possibly different joins than in all of $\mathcal{L}$.)

While all of this could have been done in the setting of $\mathcal{D}_{T}$
(and actually gives some simplified proofs for various cases) no new
theorems, of course, can be proven this way. On the other hand, given the
known counterexamples in $\mathcal{D}_{c}$, it is clear that the methods
cannot carry over to that setting. Surprisingly, all the recursion theoretic
arguments can be carried over and all the uses of fusion other than to
preserve $\omega _{1}^{CK}$ can be avoided. (Those used to show that
deciding each sentence is dense are, of course, unnecessary. Those to
convert reductions to ones resembling Turing ones can be replaced by local
Cohen forcing inside the tree parts of the conditions.) Thus it must be that
the forcing notion in the set theoretic case analogous to ours, but using
constructible trees instead of hyperarithmetic ones, fails to preserve $%
\omega _{1}$ despite the fact that ours preserves $\omega _{1}^{CK}$.

Before beginning the detailed proofs of our results, we also want to make a
comment on the coding methods used. While early results using codings in
degree structures in general (and later ones for $\mathcal{D}_{c}$) used
distributive lattices, later stronger ones have all used nondistributive
ones. It turns out that we now have a formal proof of the necessity of
moving to nondistributive lattices to get the best results. While one can
code arbitrary sets into distributive lattices (e.g. by the venerable lines
and diamonds method) the decoding takes a number of quantifiers. In degree
structures, however, while join is an effective operation (on indices) meet
is not and requires extra quantifiers. Thus to get the best possible results
one wants a coding for which the decoding is r.e. in only $\leq $ and $\vee $
(with only positive occurrences of $\leq $ as negative ones also add to the
complexity). This is the type of coding used in the many applications of
effective successor models from Shore [1981] to [2008]. As explicitly stated
in Shore [2007] and [2008] one has a recursive sequence $\phi _{n}$ of
positive $\Sigma _{1}$ formulas in $\leq $ and $\vee $ such that given any
set $X$ there is a countable lattice $\mathcal{L}_{X}$ (even one recursive
in $X$) such that $n\in X\Leftrightarrow \mathcal{L}_{X}\models \phi _{n}$.
Our results show that there is no such coding procedure in distributive
lattices. Indeed, there is an $X$ (actually the complement of $\mathcal{O}$)
such that there is no hyperarithmetic sequence $\phi _{n}$ of positive $%
\Sigma _{1}$ formulas in $\leq $ and $\vee $ and no distributive (or even
locally finite) countable lattice $\mathcal{L}_{X}$ such that $n\in
X\Leftrightarrow \mathcal{L}_{X}\models \phi _{n}$. To see this, suppose
there were such an $\mathcal{L}$. Take an embedding of $\mathcal{L}$ as an
initial segment of $\mathcal{D}_{h}$ with top the degree of some $G$. As $%
\leq _{h}$ and $\vee $ are $\Pi _{1}^{1}(G)$ on (the indices of) the initial
segment of $\mathcal{D}_{h}$ below $\deg _{h}(G)$, $X$, the complement of $%
\mathcal{O}$ would be $\Pi _{1}^{1}$ in $G$ and so $\mathcal{O}$ would be
hyperarithmetic in $G$ but the hyperdegrees below $\mathcal{O}$ do not form
a lattice. On the other hand, our coding in Proposition \ref{uslcode} shows
that there is such a recursive sequence $\phi _{n}$ such that given any $X$
we can find an usl $\mathcal{K}_{X}$ which is a susl of a locally finite
lattice $\mathcal{L}_{X}$ (each recursive in $X$) such that $n\in
X\Leftrightarrow \mathcal{K}_{X}\models \phi _{n}$. In a slightly different
vein, it seems that using the coding in distributive lattices found in
Selivanov [1988] one can code arbitrary sets into distributive lattices in
this way if one gives up the requirement that the sentences $\phi _{n}$ are
positive. (Of course, this would increase the absolute complexity of the
decoding procedure for $\mathcal{D}_{T}$ or $\mathcal{D}_{h}$.)

Our plan now is to present the definitions of usl representations and our
forcing notion in \S \ref{uslrep}. We also prove some basic facts and state
the representation theorem we need. The actual proof of the theorem we need
is given in \S \ref{repsec}. The crucial fusion lemma is proven in \S \ref%
{fusionsec} while the lemmas needed for the construction of our generic
sequence and the verifications that it gives the desired initial segment of $%
\mathcal{D}_{h}$ are in \S \ref{versec}. Finally, we close in \S \ref{quest}
with a discussion of the situation for countable usl initial segments of $%
\mathcal{D}_{h}$ and various open questions.

\section{Usl representations and the notion of forcing\label{uslrep}}

We are given a hyperarithmetic lattice $\mathcal{L}$ and a sublattice $%
\mathcal{K}$ (both with the same $0$ and $1$). We begin our journey to the
required notion of forcing with the definition of an uppersemilattice (usl)
representation.

\begin{definition}
\label{uslrepdef}Let $\Theta $ be a set of maps from an usl $\mathcal{L}$
into $\mathbb{\omega }$. For $\alpha ,\beta \in \Theta $ and $x\in \mathcal{L%
}$, we write $\alpha \equiv _{x}\beta $ ($\alpha $ is congruent to $\beta $
modulo $x$) if $\alpha (x)=\beta (x)$. We write $\alpha \equiv _{x,y}\beta $
to indicate that $\alpha $ is congruent to $\beta $ modulo both $x$ and $y$
and generally use commas conjunctively in this way. Such a $\Theta $ is an 
\emph{usl representation} of $\mathcal{L}$ if it contains the function that
is $0$ on every input and for every $\alpha ,\beta \in \Theta $ and $%
x,y,z\in \mathcal{L}$ the following properties hold:

\begin{enumerate}
\item $\alpha (0)=0$.

\item (Differentiation) If $x\nleq y$ then there are $\gamma ,\delta \in
\Theta $ such that $\gamma \equiv _{y}\delta $ but $\gamma \not\equiv
_{x}\delta $.

\item (Order) If $x\leq y$ and $\alpha \equiv _{y}\beta $ then $\alpha
\equiv _{x}\beta $.

\item (Join) If $x\vee y=z$ and $\alpha \equiv _{x,y}\beta $ then $\alpha
\equiv _{z}\beta $.
\end{enumerate}
\end{definition}

\begin{notation}
\label{restriction}If $\Theta $ is an usl representation for $\mathcal{L}$
and $\mathcal{\hat{L}}\subseteq \mathcal{L}$ then we denote the \emph{%
restriction} of $\Theta $ to $\mathcal{\hat{L}}$ by $\Theta \upharpoonright 
\mathcal{\hat{L}}=\{\alpha \upharpoonright \mathcal{\hat{L}}|\alpha \in
\Theta \}$. We also say that $\Theta $ is an \emph{extension }of $\Theta
\upharpoonright \mathcal{\hat{L}}$.
\end{notation}

For those familiar with lattice representations as presented in lattice
theory we note that these representations (with the additional requirements
for infimum given below in Theorem \ref{repthm} (2)) essentially correspond
to ones of the dual lattice by equivalence relations. If the reader prefers
to think in terms of equivalence relations on a set then the set is $\Theta $%
, the relation that corresponds to $x\in \mathcal{L}$ is $\equiv _{x}$ and,
for $\alpha \in \Theta $, its $x$-equivalence class is $\alpha (x)$. A
further reduction making the set $\Theta $ into one of natural numbers can
be achieved by identifying the elements $\alpha $ of $\Theta $ with their
values at the $1$ of $\mathcal{L}$ (and so some natural number $\alpha (1)$)
as this uniquely determines all the values of $\alpha $ by the order
property. We find this last identification that views the elements of $%
\mathcal{L}$ and those of $\Theta $ as of the same type to be confusing. (It
is really the equivalence classes that correspond to the elements of $%
\mathcal{L}$.) We also find the functional notation convenient for the
construction of representations in \S 5.

\begin{definition}
\label{homo}If $\Theta ^{\prime }$ and $\Theta $ are usl representations for 
$\mathcal{L}^{\prime }$ and $\mathcal{L}$, respectively, $\mathcal{\hat{L}}%
\subseteq \mathcal{L}^{\prime }\subseteq \mathcal{L}$ and $f:\Theta ^{\prime
}\rightarrow \Theta $, then $f$ is an $\mathcal{\hat{L}}$-\emph{homomorphism}
if, for all $\alpha ,\beta \in \Theta ^{\prime }$ and $x\in \mathcal{\hat{L}}
$, $\alpha \equiv _{x}\beta \Rightarrow f(\alpha )\equiv _{x}f(\beta )$.
\end{definition}

In \S \ref{repsec} we will prove the existence of a special type of
representation for our given lattice $\mathcal{L}$ that we need to define
our forcing conditions.\medskip

\noindent \textbf{Theorem \ref{repthm}. }\emph{If }$\mathcal{L}$\emph{\ is a
countable lattice then there is an usl representation }$\Theta $ \emph{of} $%
\mathcal{L}$ \emph{along with a nested sequence of finite slsls }$\mathcal{L}%
_{i}$\emph{\ starting with }$\mathcal{L}_{0}=\{0,1\}$ \emph{with union }$%
\mathcal{L\ }$\emph{and a nested sequence of finite subsets }$\Theta _{i}$ 
\emph{with union }$\Theta $ \emph{with both sequences recursive in }$%
\mathcal{L}$\emph{\ with the following properties:}

\begin{enumerate}
\item \emph{For each }$i\emph{,}$ $\Theta _{i}\upharpoonright \mathcal{L}%
_{i} $ \emph{is an usl representation of }$\mathcal{L}_{i}$\emph{.}

\item \emph{There are }meet interpolants\emph{\ for }$\Theta _{i}$\emph{\ in 
}$\Theta _{i+1}$\emph{, i.e.\ if }$\alpha \equiv _{z}\beta $\emph{, }$%
x\wedge y=z$\emph{\ (with }$\alpha ,\beta \in \Theta _{i}$ \emph{and }$%
x,y,z\in \mathcal{L}_{i}\emph{)}$\emph{\ then there are }$\gamma _{0},\gamma
_{1},\gamma _{2}\in \Theta _{i+1}$\emph{\ such that }$\alpha \equiv
_{x}\gamma _{0}\equiv _{y}\gamma _{1}\equiv _{x}\gamma _{2}\equiv \beta $%
\emph{.}

\item \emph{For every sublowersemilattice }$\mathcal{\hat{L}}$ \emph{of} $%
\mathcal{L}_{i}$\emph{, }$\mathcal{\hat{L}}\subseteq _{lsl}\mathcal{L}_{i}%
\emph{,}$\emph{\ there are }homogeneity interpolants\emph{\ }for $\Theta
_{i} $\ with respect to $\mathcal{\hat{L}}$\emph{\ in }$\Theta _{i+1}$\emph{%
, i.e. for every }$\alpha _{0},\alpha _{1},\beta _{0},\beta _{1}\in \Theta
_{i} $\emph{\ such that }$\forall w\in \mathcal{\hat{L}}(\alpha _{0}\equiv
_{w}\alpha _{1}\rightarrow \beta _{0}\equiv _{w}\beta _{1})$\emph{, there
are }$\gamma _{0},\gamma _{1}\in \Theta _{i+1}$\emph{\ and }$\mathcal{\hat{L}%
}$\emph{-homomorphisms }$f,g,h:\Theta _{i}\rightarrow \Theta _{i+1}$\emph{\
such that }$f:\alpha _{0},\alpha _{1}\mapsto \beta _{0},\gamma _{1}$\emph{, }%
$g:\alpha _{0},\alpha _{1}\mapsto \gamma _{0},\gamma _{1}$\emph{\ and }$%
h:\alpha _{0},\alpha _{1}\mapsto \gamma _{0},\beta _{1}$\emph{, i.e. }$%
f(\alpha _{0})=\beta _{0}$\emph{, }$f(\alpha _{1})=\gamma _{1}$ \emph{etc.}
\end{enumerate}

We point out that the presentations in Lerman [1983] and Lachlan and Lebeuf
[1976] are phrased in terms of there being only one element/two functions
homogeneity interpolants instead of the two/three we required in our
definition. One/two are not sufficient in general and two/three are used in
the original proof for finite lattices in Lerman [1971] and ones of size $%
\aleph _{1}$ in Abraham and Shore [1986]. In addition, we have changed the
definition of homogeneity interpolants from the one used in Lerman [1971]
and elsewhere in the literature by requiring that $g:\alpha _{0},\alpha
_{1}\longmapsto \gamma _{0},\gamma _{1}$ (in place of $\gamma _{1},\gamma
_{0}$). This variation is a different special case than Lerman's of a more
general definition introduced in Kjos-Hanssen [2002], [2003].
(Kjos-Hanssen's version allows a finite sequence of interpolants $\gamma
_{i} $ and functions $f_{i}$ and is indifferent to order in that sense that
it only requires that $\{f_{i}(\alpha _{0}),f_{i}(\alpha _{1}\}\}=\{\gamma
_{i},\gamma _{i+1}\}$.) Adopting our version does not makes the construction
of the required interpolants (Proposition \ref{hominterp}) any more
difficult but makes our construction of the crucial splitting trees in
Proposition \ref{splittree} considerably simpler than the previous ones.
(The same simplification would carry over to $\mathcal{D}_{T}$ and other
degree structures.) All the previous constructions proceed by using only
susls and usl representations. We have mixed in the full lattice structure
by requiring the decomposition of $\mathcal{L}$ to be into slsls instead
while still using usl representations. We exploit the full lattice structure
at various points in our construction. We will see, moreover, in \S \ref%
{quest} that this use of the meet structure of $\mathcal{L}$ is also
necessary (in contrast to the situation in $\mathcal{D}_{T}$). Given such $%
\mathcal{L}_{i}$ and $\Theta _{i}$ we next need to define the class of trees
that will be eligible to be the first component of our forcing conditions.
Each of our trees will be a function from finite strings to finite strings
(of elements from the $\Theta _{i}$) with various properties.

\begin{definition}
\label{tree}A \emph{tree }$T$\emph{\ (for the sequence }$\langle \mathcal{L}%
_{i},\Theta _{i}\rangle $\emph{)} has an associated natural number $k=k(T)$. 
$T$ is a hyperarithmetic function such that for some $k\in \omega $ it has
the empty string $\emptyset $ and all strings in the Cartesian product $%
\prod\limits_{n=0}^{n=m}\Theta _{k+n}$ for each $m\in \mathbb{\omega }$. We
denote this number $k$ by $k(T)$. Moreover, for each $\sigma \in \text{dom }T$%
, $T(\sigma )\in \prod\limits_{n=0}^{n=q}\Theta _{n}$ for some $q\geq
|\sigma |-1$. Moreover, $T$ has the following properties for all $\sigma
,\tau \in \text{dom }T$:

\begin{enumerate}
\item (Order) $\sigma \subseteq \tau \Rightarrow T(\sigma )\subseteq T(\tau
) $.

\item (Nonorder) $\sigma |\tau \Rightarrow T(\sigma )|T(\tau )$. (We use $|$
to denote incompatibility of strings.)

\item (Uniformity) For every fixed length $l$ there is, for each $\alpha \in
\Theta _{k+l}$, a string $\rho _{l,\alpha }$ so that, for a given $l$, all
the $\rho _{l,\alpha }$ are of the same length independently of $\alpha $
and if $|\sigma |=l$ then $T(\sigma \symbol{94}\alpha )=T(\sigma )\symbol{94}%
\rho _{l,\alpha }$. (We use $\symbol{94}$ to denote concatenation and
confuse a single symbol such as $\alpha $ with the string $\left\langle
\alpha \right\rangle $ of length one.)
\end{enumerate}
\end{definition}

Thus our trees $T$ have branchings of width $|\Theta _{k(T)+n}|$ at level $n$
and obey the usual order and uniformity properties of those used in Lerman
[1983] to prove initial segment results for $\mathcal{D}_{T}$. Our
definition of $S$ being a subtree of $T$ incorporates the (technically
convenient) requirement also used there that the branchings on $S$ follow
those on $T$.

\begin{definition}
\label{subtree}We say that a tree $S$ is a \emph{subtree} of a tree $T$, $%
S\subseteq T$, if $k(S)\geq k(T)$ and $(\forall \sigma \in \text{dom }%
S)(\exists \tau \in \text{dom }T)[S(\sigma )=T(\tau )~\&~(\forall \alpha \in
\Theta _{k(S)+|\sigma |})(S(\sigma \symbol{94}\alpha )\supseteq T(\tau 
\symbol{94}\alpha ))]$.
\end{definition}

If it does not seem obvious, transitivity of the subtree relation will be
proven in Proposition \ref{trans}. We mention some specific operations on
trees that we will need later.

\begin{definition}
\label{fulldef}If $T$ is a tree and $\sigma \in \text{dom }T$ then $T_{\sigma
}$ is defined by $T_{\sigma }(\tau )=T(\sigma \symbol{94}\tau )$. Clearly, $%
k(T_{\sigma })=k(T)+|\sigma |$ and $T_{\sigma }\subseteq T$. For a string $%
\mu \in \prod\limits_{n=0}^{n=q}\Theta _{n}$ with $q\leq |T(\emptyset )|-1$,
we let $T^{\mu }$ (the \emph{transfer tree} of $T$ over $\mu $) be the tree
such that, for every $\sigma \in \text{dom }T$, $T^{\mu }(\sigma )$ is the
string gotten from $T(\sigma )$ by replacing its initial segment of length $%
q+1$ (which is contained in $T(\emptyset )$) by $\mu $. We write $T_{\sigma
}^{\mu }$ for $(T_{\sigma })^{\mu }$.
\end{definition}

A crucial notion for our constructions is that of preserving the congruences
of specified slsls of our given lattice $\mathcal{L}$.

\begin{definition}
\label{prescong}If $\mathcal{\hat{L}}$ is a finite slsl of $\mathcal{L}$ we
say that a subtree $S$ of $T$ \emph{preserves the congruences of }$\mathcal{%
\hat{L}}$, $S\subseteq _{\mathcal{\hat{L}}}T$, if $\mathcal{\hat{L}}%
\subseteq \mathcal{L}_{k(T)}$ and, whenever $x\in \mathcal{\hat{L}}$, $%
S(\sigma )=T(\tau )$, $\alpha \equiv _{x}\beta $, $S(\sigma \symbol{94}%
\alpha )=T(\tau \symbol{94}\mu )$ and $S(\sigma \symbol{94}\beta )=T(\tau 
\symbol{94}\nu )$, then $\mu \equiv _{x}\nu $. Here $\alpha $ and $\beta $
are members of the appropriate $\Theta _{i}$ and $\mu $ and $\nu $ are
sequences (necessarily of the same length $m$) of elements from the
appropriate $\Theta _{j}$'s. We say that such sequences $\mu $ and $\nu $
are congruent modulo $x$, $\mu \equiv _{x}\nu $, if $\mu (j)\equiv _{x}\nu
(j)$ for each $j<m$.
\end{definition}

\begin{proposition}
\label{trans}If $R\subseteq _{\mathcal{L}_{1}}S\subseteq _{\mathcal{L}_{2}}T$
and then $R\subseteq _{\mathcal{L}_{1}\cap \mathcal{L}_{2}}T$.
\end{proposition}

\begin{proof}
To see that $R\subseteq T$ note first that $k(R)\geq k(S)\geq k(T)$. Next
suppose that $\rho \in \text{dom }R$ and $\alpha \in \Theta _{k(R)+|\rho |}$.
As $R\subseteq S$ we have a $\sigma $ such that $R(\rho )=S(\sigma )$ and $%
R(\rho \symbol{94}\alpha )\supseteq S(\sigma \symbol{94}\alpha ))$. As $%
S\subseteq T$ we have a $\tau $ such that $S(\sigma )=T(\tau )$ and $%
S(\sigma \symbol{94}\alpha )\supseteq T(\tau \symbol{94}\alpha ))$. Thus $%
R(\rho )=T(\tau )$ and $R(\rho \symbol{94}\alpha )\supseteq T(\tau \symbol{94%
}\alpha )$ as required. As for the preservation of $\mathcal{L}_{1}\cap 
\mathcal{L}_{2}$ congruences, suppose $R(\rho )=S(\sigma )=T(\tau )$, $x\in 
\mathcal{L}_{1}\cap \mathcal{L}_{2}$, $\alpha _{0},\alpha _{1}\in \Theta
_{k(R)+|\rho |}$ and $\alpha _{0}\equiv _{x}\alpha _{1}$. Let $R(\rho 
\symbol{94}\alpha _{i})=S(\sigma \symbol{94}\mu _{i})=T(\tau \symbol{94}\nu
_{i})$. As $x\in \mathcal{L}_{1}$ and $R\subseteq _{\mathcal{L}_{1}}S$, $\mu
_{0}\equiv _{x}\mu _{1}$. As $x\in \mathcal{L}_{2}$ and $S\subseteq _{%
\mathcal{L}_{2}}T$ it then follows by induction on the (by uniformity,
necessarily common) length of $\mu _{i}$ that $\nu _{0}\equiv _{x}\nu _{1}$
as required. (Write $\nu _{i}=\nu _{i}^{0}\symbol{94}\cdots \symbol{94}\nu
_{i}^{s}$ where $S(\sigma \symbol{94}\mu _{i}(0)\cdots \symbol{94}\mu
_{i}(t))=T(\tau \symbol{94}\nu _{i}^{0}\symbol{94}^{\cdots }\symbol{94}\nu
_{i}^{t})$. Then inductively $\mu _{0}(j)\equiv _{x}\mu _{1}(j)$ gives $\nu
_{0}^{j}\equiv _{x}\nu _{1}^{j}$.)
\end{proof}

\begin{definition}
\label{defforcing}Recall that $\mathcal{K}$ is an arbitrary sublattice of $%
\mathcal{L}$ (preserving $0$ and $1$). Our notion of forcing $\mathcal{P}%
_{\langle \mathcal{L}_{i},\Theta _{i}\rangle ,\mathcal{K}}=\mathcal{P}$ is
given first $_{{}}$by letting the \emph{forcing conditions} $P$ be the pairs 
$\langle T,\mathcal{\hat{K}}\rangle $ where $T$ is a tree for $\langle 
\mathcal{L}_{i},\Theta _{i}\rangle $ and $\mathcal{\hat{K}}$ is a finite
slsl of $\mathcal{K}\cap \mathcal{L}_{k(T)}$. We then say that $\langle
T_{1},\mathcal{K}_{1}\rangle $ \emph{extends }or\emph{\ refines} $\langle
T_{0},\mathcal{K}_{0}\rangle $, $\langle T_{1},\mathcal{K}_{1}\rangle \leq _{%
\mathcal{P}}$ $\langle T_{0},\mathcal{K}_{0}\rangle $ if $T_{1}\subseteq _{%
\mathcal{K}_{0}}T_{0}$ and $\mathcal{K}_{1}\supseteq \mathcal{K}_{0}$. If $%
P=\langle T,\mathcal{\hat{K}}\rangle $ is a condition we let $K(P)=\mathcal{%
\hat{K}}$, $Tr(P)=T$ and $k(P)=k(T)$. If the notion of forcing is fixed in
some context we will often omit the subscript $\mathcal{P}$ in $\leq _{%
\mathcal{P}}$. We sometimes abuse notation by identifying a condition $P$
with $Tr(P)$ when $K(P)$ is fixed. Along these lines, for example, we use $%
P_{\sigma },P^{\tau }$ and $P_{\sigma }^{\tau }$ to stand for $\langle
Tr(P)_{\sigma },K(P)\rangle ,\langle Tr(P)^{\tau },K(P)\rangle $ and $%
\langle Tr(P)_{\sigma }^{\tau },K(P)\rangle $, respectively. The top element
of $\mathcal{P}$ consists of the identity tree $Id$ (which has $k(Id)=0$)
and the slsl $\mathcal{L}_{0}=\{0,1\}$.
\end{definition}

\begin{lemma}
\label{fulllem}If $T$ is a tree, $\sigma \in \text{dom }T$ and $\mathcal{\hat{%
L}}\subseteq \mathcal{L}_{k(T)}$, then $T_{\sigma }\subseteq _{\mathcal{\hat{%
L}}}T$. So if $P=\langle T,\mathcal{\hat{K}}\rangle $ is a condition $%
\langle T_{\sigma },\mathcal{\hat{K}}\rangle \leq _{\mathcal{P}}\langle T,%
\mathcal{\hat{K}}\rangle $ (by letting $\mathcal{\hat{L}}=\mathcal{\hat{K}}$%
). If $\sigma ,\tau \in \text{dom }T$ are of the same length and $\langle S,%
\mathcal{\hat{K}}\rangle \leq $ $\langle T_{\sigma },\mathcal{\hat{K}}%
\rangle $ then $\langle S^{T(\tau )},\mathcal{\hat{K}}\rangle \leq _{%
\mathcal{P}}\langle T_{\tau },\mathcal{\hat{K}}\rangle $. We also have that $%
T_{\sigma }^{T(\tau )}=T_{\tau }$.
\end{lemma}

\begin{proof}
The first assertions follow directly from the definitions. The last two
follow from the uniformity assumption on our trees.
\end{proof}

It is easy to see that with trivial genericity requirements any generic
filter $\mathcal{G}$ determines a function $G\in
D=\prod\limits_{n=0}^{\infty }\Theta _{n}$ , i.e. a function on $\omega $
with $G(n)\in \Theta _{n}$. ($G=\bigcup \{T(\emptyset )|T$ is the first
component of some condition in $\mathcal{G}\}$.) On this basis we could
naively try to define our embedding of $\mathcal{K}$ into the hyperdegrees
as follows. For $x\in \mathcal{K}\subseteq \mathcal{L}$ we let $G_{x}:\omega
\rightarrow \omega $ be defined by $G_{x}(n)=G(n)(x)$. (As $G(n)\in \Theta
_{n}\subseteq \Theta $ it is a map from $\mathcal{L}$ into $\omega $.) The
desired image of $x$ would then be $\deg _{h}(G_{x})$. Now the order and
join properties of usl representations guarantee that this embedding
preserves order and join (on all of $\mathcal{L}$ even). If $x\leq y$ then
by the order property we can (recursively in the hyperarithmetic
representation $\left\langle \Theta _{i}\right\rangle $) calculate $G_{x}(m)$
from $G_{y}(m)$ by finding any $\alpha \in \Theta _{m}$ with $\alpha
(y)=G_{y}(m)$ and declaring that $G_{x}(m)=\alpha (x)$. (Such an $\alpha $
exists since $G(m)$ is one.) Similarly if $x\vee y=z$ then, by the join
property, we can calculate $G_{z}(m)$ from $G_{x}(m)$ and $G_{y}(m)$ by
finding any $\alpha \in \Theta _{m}$ such that $\alpha (x)=G_{x}(m)$ and $%
\alpha (y)=G_{y}(m)$ and declaring that $G_{z}(m)=\alpha (z)$. (Again $G(m)$
is such an $\alpha $.)

Were congruences modulo $x$ always preserved for every $x$, we could
directly carry out the diagonalization and other requirements as well for
this definition of $G_{x}$. (Of course, we cannot allow this to happen as it
would produce an embedding of all of $\mathcal{L}$ as an initial segment of $%
\mathcal{D}_{h}$ in place of the one wanted of $\mathcal{K}$.) In actuality,
not all congruences are preserved as we refine to various subtrees in our
construction. Thus we must modify the definition of the images in $\mathcal{D%
}_{h}$ and provide nice representations of the hyperdegrees corresponding to 
$x$. To that end we introduce some dense sets that our generic must meet.

\begin{lemma}
\label{basicden}For each $x\in \mathcal{K}$ and $k\in \mathbb{\omega }$ the
sets $\{P|x\in K(P)\}$ and $\{P|k(P)\geq k\}$ are dense.
\end{lemma}

\begin{proof}
Consider any $Q\in \mathcal{P}$, $x\in \mathcal{K}$ and $k\in \omega $. Let $%
\mathcal{K}^{\prime }$ be the slsl of $\mathcal{K}$ generated by $K(Q)$ and $%
x$ and let $i\geq k(Q),k$ be such that $\mathcal{K}^{\prime }\subseteq 
\mathcal{L}_{i}$. Define $S$ with $k(S)=i$ by $S(\sigma )=Tr(Q)(0^{i-k(Q)}%
\symbol{94}\sigma )$. Clearly, $\langle S,\mathcal{K}^{\prime }\rangle \leq
_{\mathcal{P}}Q$ and is in both required sets. \bigskip
\end{proof}

From now on assume that any generic filter $\mathcal{G}$ we consider meets
the dense sets of Lemma \ref{basicden}. They suffice to define our generic $%
G $ as above. More crucially, they allow us to define nice representatives
of the desired images of $x$ in $\mathcal{D}_{h}$.

\begin{definition}
If $\mathcal{G}$ is a generic filter meeting the dense sets of Lemma \ref%
{basicden}, $G$ the corresponding element of $\prod\limits_{n=0}^{\infty
}\Theta _{n}$, $P\in \mathcal{G}$ and $x\in K(P)$ then $G^{P}$ is the
sequence $\left\langle \alpha _{n}|n\in \omega \right\rangle $ where $%
Tr(P)(\left\langle \alpha _{n}|n<m\right\rangle )\subseteq G$ for every $m$.
(Thus $\left\langle \alpha _{n}\right\rangle $ is the path that $G$ follows
in the domain of $Tr(P)$. It is obvious from the definitions that $G$ is a
path on (i.e. in the range of) $Tr(Q)$ for every $Q\in \mathcal{G}$.) We
define $G_{x}^{P}(n)$ as $\alpha _{n}(x)$. (In the parlance of Lachlan and
Lebeuf [1976], $G^{P}$ is the \emph{signature of }$G$\emph{\ on }$Tr(P)$.)
\end{definition}

The crucial point is that the hyperdegree of $G_{x}^{P}$ does not depend on $%
P$.

\begin{lemma}
If $x\in K(P),K(Q)$ for $P,Q$ in a generic $\mathcal{G}$, then $%
G_{x}^{P}\equiv _{h}G_{x}^{Q}$.
\end{lemma}

\begin{proof}
As there is an $R\leq P,Q$ in $\mathcal{G}$ by the compatibility of all
conditions in a generic filter, it suffices to consider the case that $Q\leq
P$. Let $G^{P}=\left\langle \alpha _{n}\right\rangle $ and $%
G^{Q}=\left\langle \beta _{n}\right\rangle $. By the definition of subtree
there is for each $n$ an $m(n)$ such that $Tr(Q)(\left\langle \beta
_{s}|s<n\right\rangle )=Tr(P)(\left\langle \alpha _{m(s)}|s<n\right\rangle )$
and we can compute the function $m$ recursively in the trees. (By the
uniformity of the trees there is for each $n$ a unique $m(n)$ such that $%
|Tr(Q)(\sigma )|=|Tr(P)(\tau )|$ for every $\sigma $ of length $n$ and every 
$\tau $ of length $m(n)$.) Moreover, by our definition of subtree, $\beta
_{n+1}=\alpha _{m(n)+1}$. Thus $G_{x}^{Q}(n)=\beta _{n}(x)=\alpha
_{m(n)}(x)=G_{x}^{P}(m(n))$ and so $G_{x}^{Q}\leq _{h}G_{x}^{P}$. The other
direction depends on the congruence preservations for $x$ implied by $%
Tr(Q)\subseteq _{K(P)}Tr(P)$.

Suppose that we have, by a hyperarithmetic recursion, determined $%
G_{x}^{P}(i)=\alpha _{i}(x)$ for $i\leq m(n)$. The next step followed by $G$
in $Q$ is $\beta _{n+1}=\alpha _{m(n)+1}$. It corresponds to the sequence $%
\left\langle \alpha _{i}|m(n)+1\leq i<m(n+1)\right\rangle $. The definition
of $\subseteq _{K(P)}$ implies that $\left\langle \alpha _{i}(x)|m(n)+1\leq
i<m(n+1)\right\rangle $ is uniquely determined by $\beta _{n+1}(x)$ to
continue the recursion.
\end{proof}

Thus given a generic $\mathcal{G}$ we can define a map from $\mathcal{K}$
into $\mathcal{D}_{h}$ by sending $x\in \mathcal{K}$ to $\deg
_{h}(G_{x}^{P}) $ for any $P\in \mathcal{G}$ with $x\in K(P)$. Our naive
proofs of the preservation of order an join can now be made precise by
simply applying them to $G^{P}$ on $Tr(P)$ (in place of $G$ on $Id$) for any 
$P\in \mathcal{G}$ with $x,y,z\in K(P)$.

Now we must define, analyze and exploit our forcing language and relation to
prove that the embedding corresponding to a sufficiently generic $G$
preserves nonorder and produces an initial segment of the hyperdegrees.

We proceed essentially as for (perfect) forcing with binary trees in Sacks
[1990, IV.4]. Our base structure (model) is formally the ramified analytic
hierarchy up to $\omega _{1}^{CK}$ (Church-Kleene $\omega _{1}$, the first
nonrecursive ordinal) which we denote by $\mathcal{M}$. This is the analog
of $L_{\omega _{1}^{CK}}$ where the language and definitions are set in
second order arithmetic instead of set theory. The sets in the structure
are, however, just those subsets of $\omega $ that appear in $L_{\omega
_{1}^{CK}}$ or equivalently the hyperarithmetic ones. Our forcing extension
will be $\mathcal{M}(\omega _{1}^{CK},G)$ which is defined analogously to $%
L_{\omega _{1}^{CK}}[G]$. If $\omega _{1}^{G}$, the least ordinal not
recursive in $G$, is $\omega _{1}^{CK}$ (as will be the case for generic $G$%
) then the sets in $\mathcal{M}(\omega _{1}^{CK},G)$ are precisely those
hyperarithmetic in $G$. To define and describe the model and forcing
relation we have a ramified forcing language as introduced in Feferman
[1965] and described in Sacks [1990, III.4] with a term $\mathcal{G}$ for
the function $G$ (in place of $\mathcal{T}$ and the set $T$) to describe
this structure and define the forcing relation. (We are thinking of
everything such as $\mathcal{L}$, $\mathcal{K}$, $\mathcal{L}_{i}$ and $%
\Theta _{i}$ as coded in second order arithmetic.) In addition to the usual
paraphernalia of second order arithmetic, the language has ranked set
variables $X^{\zeta }$ for ordinals $\zeta <\omega _{1}^{CK}$ which range
over the sets constructed by level $\zeta $ in $\mathcal{M}(\omega
_{1}^{CK},G)$. A formula is ranked if all its set variables are ranked. We
refer to Sacks [1990, III.4] for the simultaneous recursive definitions of
the sets $\mathcal{M}(\zeta ,G))$, interpretations of terms $\hat{x}\mathcal{%
H}(x)$ for the set of numbers satisfying the ranked formula $\mathcal{H}$
and for the truth of ranked and unranked formulas in the structure $\mathcal{%
M}(\omega _{1}^{CK},G)$. The definitions proceed through inductions on a
reasonably defined notion of \textquotedblleft full ordinal
rank\textquotedblright\ of a formula. Sacks also provides an analysis of the
complexity of the satisfaction relation for ranked (and unranked) formulas.
Of course, unranked formulas are interpreted as having their unranked
variables ranging over all of $\mathcal{M}(\omega _{1}^{CK},G)$. As the
change from his setting to ours is purely notational, going from a set
(element of $2^{\omega }$) to a function in $D$, the development there
carries over with only notational changes.

The crucial starting point of the definition of the forcing is that $\langle
T,\mathcal{\hat{K}}\rangle \Vdash \mathcal{F}$ for a ranked formula $%
\mathcal{F}$ if and only if $\mathcal{M}(\omega _{1}^{CK},G)\models \mathcal{%
F}$ for every $G\in \lbrack T]$ (the paths through $T$). The relation is
then defined on the unranked formulas by induction on the usual complexity
of formulas. The clauses for conjunction and negation are standard: $P\Vdash 
\mathcal{F}\wedge \mathcal{H}\Leftrightarrow P\Vdash \mathcal{F}$ and $%
P\Vdash \mathcal{H}$; $P\Vdash \lnot \mathcal{F}\Leftrightarrow (\forall
Q\leq P)(Q\nVdash \mathcal{F})$. There are three cases for the existential
quantifiers: $P\Vdash \exists x\mathcal{F}(x)\Leftrightarrow P\Vdash 
\mathcal{F}(\underline{n})$ for some numeral $\underline{n}$; $P\Vdash
\exists X^{\zeta }\mathcal{F}(X^{\zeta })$ for $\zeta <\omega _{1}^{CK}$ $%
\Leftrightarrow P\Vdash \mathcal{F}(\hat{x}\mathcal{H}(x))$ for some $%
\mathcal{H}$ of rank at most $\zeta $; $P\Vdash \exists X\mathcal{F}%
(X)\Leftrightarrow P\Vdash \exists X^{\zeta }\mathcal{F}(X^{\zeta })$ for
some $\zeta <\omega _{1}^{CK}$.

The key complexity theoretic fact from Sacks [1990, IV.4] is that the
forcing relation for $\Sigma _{1}^{1}$ sentences (those with initial
unranked existential set quantifiers followed by a ranked formula) is a $\Pi
_{1}^{1}$ relation. This fact is again established quite abstractly and is
still true by the same simple argument for our new definition of forcing
conditions. The crucial property of perfect forcing in Sacks [1990] that now
allows one to prove that for generic $G$ forcing equals truth and $\omega
_{1}^{CK}$ is preserved (i.e.\ $\omega _{1}^{G}=\omega _{1}^{CK}$) is the
fusion lemma to which we now turn.

\section{The fusion lemma\label{fusionsec}}

We can view the problem of preserving $\omega _{1}^{CK}$ as preserving
admissibility in $L_{\omega _{1}^{CK}}$ when we add on the new function $G$.
From this point of view, we want to preserve $\Sigma _{1}$ replacement,
i.e.\ $L_{\omega _{1}^{CK}}[G]\models \forall n\in \omega \exists X\varphi
(n,X)\rightarrow \exists X\forall n\in \omega (\varphi (n,X^{[n]})$.
Translated into our setting of second order arithmetic and $\mathcal{M}%
(\omega _{1}^{CK},G)$ this is equivalent to preserving $\Delta _{1}^{1}$%
-comprehension. The key in either case is what is called the fusion lemma.
It it is a natural step from the replacement version when one replaces
satisfaction by forcing and wants to prove that one can densely decide
sentences of the desired form (as is needed to show the usual equivalence
between forcing and truth).

\begin{lemma}[Fusion]
\label{fusion}Let $P=\langle T,\mathcal{\hat{K}}\rangle $ be a forcing
condition and $\exists \bar{X}_{n}\mathcal{F}_{n}(\bar{X}_{n})$ a
hyperarithmetic sequence of $\Sigma _{1}^{1}$ sentences with their unranked
quantifiers displayed such that $\forall n\forall Q\leq P\exists R\leq
Q[R\Vdash \exists \bar{X}_{n}\mathcal{F}_{n}(\bar{X}_{n})]$. Then there is a 
$Q\leq P$, and even one with $k(Q)=k(P)$ and $K(Q)=K(P)$, such that $\forall
n(Q\Vdash \exists \bar{X}_{n}\mathcal{F}_{n}(\bar{X}_{n}))$. Moreover, there
is a $\zeta <\omega _{1}^{CK}$ such that $\forall n(Q\Vdash \exists \bar{X}%
_{n}^{\zeta }\mathcal{F}_{n}(\bar{X}_{n}^{\zeta }))$.
\end{lemma}

Now what goes wrong with an attempted proof of this theorem if we use some
standard type of sequential representations $\Theta _{i}$ for $\mathcal{L}%
_{i}$ and the associated trees? When we try to prove the fusion lemma, we
thin out our given tree (above each node successively) and get conditions
that are trees for $\Theta _{i}$ for larger and larger $i$ as we treat each $%
\mathcal{F}_{n}$ in turn. As the trees are of different shapes for each $i$
and the representations not nicely nested, there is no way to combine (fuse)
the subtrees into the single condition $Q$ required. The first crucial
property of our representation is that it allows us to use trees that are
essentially (i.e.\ eventually) of the same shape with splittings for a
cofinal segment of the $\Theta _{i}$. We now turn to the basic lemma that
conveys the other technical facts about our notion of forcing that enables
us to overcome this problem. It was for this application that our lattice
representation theorem (Theorem \ref{repthm}) was designed.

\begin{proposition}
\label{refine}If $\langle T_{1},\mathcal{K}_{1}\rangle \leq _{\mathcal{P}}$ $%
\langle T_{0},\mathcal{K}_{0}\rangle $ then $\langle T_{1},\mathcal{K}%
_{0}\rangle \leq _{\mathcal{P}}$ $\langle T_{0},\mathcal{K}_{0}\rangle $.
Moreover, we can get an $S$ with $[S]\subseteq \lbrack T_{1}]$ such that $%
k(S)=k(T_{0})$ and $\langle S,\mathcal{K}_{0}\rangle \leq _{\mathcal{P}}$ $%
\langle T_{0},\mathcal{K}_{0}\rangle $. Thus, if $\mathcal{F}$ is a ranked
formula, $P$ is a condition, $Q\leq P$ and $Q\Vdash \mathcal{F}$, then there
is an $R\leq P$ such that $R\Vdash \mathcal{F}$, $k(R)=k(P)$ and $K(R)=K(P)$.
\end{proposition}

\begin{proof}
That $\langle T_{1},\mathcal{K}_{0}\rangle \leq _{\mathcal{P}}$ $\langle
T_{0},\mathcal{K}_{0}\rangle $ is clear from the definition of extension. To
get the desired $S$, begin with $T_{1}$. Now $k(T_{1})=k_{1}\geq
k_{0}=k(T_{0})$ and so the branchings at the $nth$ level of $T_{1}$ which
correspond to the elements of $\Theta _{k_{1}+n}$ contain the branchings in $%
T_{0}$ which correspond to the elements of $\Theta _{k_{0}+n}$ as by our
representation theorem (Theorem \ref{repthm}), $\Theta _{k_{0}+n}\subseteq
\Theta _{k_{1}+n}$. To be more precise we define $S(\emptyset
)=T_{1}(\emptyset )$ and $S(\sigma )$ for $\sigma \in $ $\prod%
\limits_{n=0}^{n=m}\Theta _{k_{0}+n}$ as $T_{1}(\sigma )$. So for every $%
\sigma $ in its domain $S(\sigma )=T_{1}(\sigma )$. It is then clear that $%
[S]\subseteq \lbrack T_{1}]$, $k(S)=k(T_{0})$ and $\langle S,\mathcal{K}%
_{0}\rangle \leq _{\mathcal{P}}$ $\langle T_{0},\mathcal{K}_{0}\rangle $ as
required. Now if $Q\leq P$ and $Q\Vdash \mathcal{F}$ with $\mathcal{F}$
ranked then this argument gives us an $R\leq P$ with $[R]\subseteq \lbrack
Q] $, $k(R)=k(P)$ and $K(R)=K(P)$. As $\mathcal{F}$ is ranked and $Q\Vdash 
\mathcal{F}$, $[R]\subseteq \lbrack Q]$ implies that $R\Vdash \mathcal{F}$
as well by the definition of forcing for ranked sentences.
\end{proof}

For those familiar with the proof of the fusion lemma in other settings, we
note how our special representation allows us to avoid changing the shape of
the trees or increasing the slsl $\mathcal{\hat{K}}$ with respect to which
we are preserving congruences. The point is that when we have some condition 
$P=\langle T_{0},\mathcal{K}_{0}\rangle $ and ask if there is any $Q=\langle
T_{1},\mathcal{K}_{1}\rangle $ extending $P$ that forces some ranked formula 
$\mathcal{F}$ we can restrict ourselves to those with $k(T_{1})=k(T_{0})$
and $\mathcal{K}_{1}=\mathcal{K}_{0}$ by Proposition \ref{refine}. This idea
is really the only new one needed to carry out the proof of the fusion lemma.

\begin{proof}[Proof of Fusion Lemma]
As the forcing relation for $\Sigma _{1}^{1}$ sentences is $\Pi _{1}^{1}$
and extension is arithmetic by definition, the relation $R\leq
Q~\&~k(R)=k(P)~\&~K(R)=K(P)~\&~R\Vdash $ $\exists \bar{X}_{n}\mathcal{F}_{n}(%
\bar{X}_{n})$ is $\Pi _{1}^{1}$ and it can be uniformized to a partial $\Pi
_{1}^{1}$ function $R(n,Q)$ by Kreisel uniformization (Sacks [1990 II.2.6]).
Our hypothesis and Proposition \ref{refine} together imply that $R$ is
defined for every $\langle n,Q\rangle $ with $Q\leq P$. We now define a tree 
$S$ by recursion so that $\langle S,\mathcal{\hat{K}}\rangle $ is the
forcing condition required in the Lemma. In fact, along with the definition
of $S(\sigma )$ by induction on the length $n$ of $\sigma $ we define an
auxiliary sequence of conditions $\langle U_{n,i},\mathcal{\hat{K}}\rangle
\leq P$ with $k(U_{n,i})=k(P)$ and $|U_{n,i}(\emptyset )|=|S(\sigma )|$
(which depends only on $n$ by uniformity) for $i\leq m(n)$ where $m(n)$ is
one less than the number of branches on $T$ at level $n$. Indeed we will
have that $\langle U_{n,i}^{T(\sigma )},\mathcal{\hat{K}}\rangle \leq P$ for
each $\sigma $ of length $n$. (Recall Definition \ref{fulldef} and Lemma \ref%
{fulllem}.) We begin with $S(\emptyset )=T(\emptyset )$. Suppose we have
defined $S(\sigma )$ for all $\sigma $ with length $n$ and have the
corresponding $U_{n,m(n)}$. For any $\sigma $ of length $n$ and any
appropriate $\alpha $ we let $S(\sigma \symbol{94}a)=U_{n,m(n)}^{S(\sigma
)}(0^{n}\symbol{94}\alpha )$. Thus $S$ (up to this level) is clearly uniform
and a subtree of $T$ that preserves the congruences in $\mathcal{\hat{K}}$.
We now define the $U_{n+1,i}$ for $i<m(n+1)$ by induction on $i$. We begin
with $\sigma _{0}=\rho \symbol{94}\alpha $ (for some $\rho $ of length $n$
and some $\alpha $ in the appropriate $\Theta _{j}$). We let $U_{n+1,0}$ be
the first component of $R(n+1,(U_{n,m(n)})_{0^{n}\symbol{94}\alpha
}^{S(\sigma _{0})})$. Given $U_{n+1,i}$ and $\sigma _{i+1}$ we let $%
U_{n+1,i+1}$ be the first component of $R(n+1,U_{n+1,i}^{S(\sigma _{i})})$.
(So, by induction, $\langle U_{n+1,i},\mathcal{\hat{K}}\rangle \leq P$ and $%
\langle U_{n+1,i},\mathcal{\hat{K}}\rangle \Vdash \mathcal{F}_{n+1}$.) This
completes the inductive definition of the $U_{n+1,i}$ and so of $S$. The
crucial fact about $S$ is that for each $n$ and $\sigma $ of length $n$
every path on $S$ that contains $S(\sigma )$ is a path on $U_{n,i}$ for some 
$i$ and so makes $\mathcal{F}_{n}$ true by our choice of $U_{n,i}$ as the
first component of value of $R(n,U)$ for some $U$ with $\langle U,\mathcal{%
\hat{K}}\rangle \leq P$. Thus $\mathcal{F}_{n}$ is true of every path on $S$
for every $n$, i.e.\ $\forall n(\langle S,\mathcal{\hat{K}}\rangle \Vdash 
\mathcal{F}_{n})$.
\end{proof}

Given the fusion lemma for our forcing, the remaining development of the
basic facts about our forcing is routine and follows Sacks [1990, IV.4]. In
particular, essentially the same proofs show first that for every $P$ and $%
\mathcal{F}$ there is a $Q\leq P$ that decides $\mathcal{F}$, i.e.\ $Q\Vdash 
\mathcal{F}$ or $Q\Vdash \lnot \mathcal{F}$ and then that for any generic
filter $\mathcal{G}$ (i.e.\ one such that for every $\mathcal{F}$ there is a 
$P\in \mathcal{G}$ that decides $\mathcal{F}$) forcing equals truth in the
sense that for, every $\mathcal{F}$, $\mathcal{M}(\omega _{1}^{CK},G)\models 
\mathcal{F}\Leftrightarrow (\exists P\in \mathcal{G})\mathbf{(}P\Vdash 
\mathcal{F})$.

\begin{remark}
One can simplify Sacks's inductive proof that forcing equals truth by
constructing a generic sequence $P_{i}$ of conditions such that each
sentence $\mathcal{F}$ of our language is decided by some $P_{i}$ with the
sentences arranged so that both $\mathcal{F}$ and $\mathcal{H}$ are decided
before deciding $\mathcal{F~}\&~\mathcal{H}$. This would allow one to avoid
having to prove an analog of Sacks's Proposition IV.4.6 as the case of a
conjunction in Sacks's inductive proof that forcing equals truth is now
immediate.
\end{remark}

In addition to meeting the dense sets of conditions deciding each $\mathcal{F%
}$, we explicitly meet certain other sets of conditions to assure that we
can define our embedding on all of $\mathcal{K}$ and show that it is a
one-one map onto the hyperdegrees below the generic $G=G_{1}^{Id}$.

\section{Initial segment verifications\label{versec}}

To simplify our notation for the sets hyperarithmetic in $G$, i.e.\ those in 
$\mathcal{M}(\omega _{1}^{CK},G)$, and to emphasize the analogy to the
Turing degree constructions of initial segments, we number the terms $\hat{x}%
\mathcal{H}$ of our language by ordinals $\delta <\omega _{1}^{CK}$ and
denote the characteristic function of the set they stand for by $\{\delta
\}^{G}$. (Note that, in contrast to the situation with Turing reductions,
these are all total functions but, of course, the list is only recursive in $%
\mathcal{O}$.) We at times use $\delta _{x}^{P}$ to name the reduction such
that $\{\delta _{x}^{P}\}^{G}=\{\delta \}^{G_{x}^{P}}$.

To assure that our embedding preserves nonorder we want to show for any $%
x\nleq y$ in $\mathcal{K}$, condition $P$ with $x,y\in K(P)$ and $\delta
<\omega _{1}^{CK}$, that $\{Q|Q\Vdash \{\delta _{y}^{P}\}^{G}\neq
G_{x}^{P}\} $ is dense below $P$. In addition we want to show that $%
\{Q|Q\Vdash \exists x\in \mathcal{K}(Q\Vdash \{\delta \}^{G}\equiv
_{h}G_{x}^{P})\}$ is also dense below any $P$ for each $\delta $. These two
results would then finish the proof or our theorem. We begin with an
auxiliary collection of dense sets that make our task much simpler. They
correspond to the total subtrees of Lerman [1983] and allow us to convert
the arguments for $\mathcal{D}_{T}$ to our setting of $\mathcal{D}_{h}$.

\begin{lemma}[Totality]
\label{tot}For each condition $P$ and $\delta <\omega _{1}^{CK}$ there is a $%
Q\leq P$ with $k(Q)=k(P)$ and $K(Q)=K(P)$ such that for every $x$ and $%
\sigma $ of length $x$ there is an $n\in \{0,1\}$ for which $Q_{\sigma
}\Vdash \{\delta \}^{G}(x)=\underline{n}$. Moreover, the map $q$ taking $%
\langle x,\sigma \rangle $ to this $n$ is hyperarithmetic.
\end{lemma}

\begin{proof}
For any $Q\leq P$ and any $m$ there is clearly an $R\leq Q$ and an $n\in
\{0,1\}$ for which $R\Vdash \{\delta \}^{G}=\underline{n}$. (Otherwise by
the density of deciding formulas we could find $R_{1}\leq R_{0}\leq Q$ such
that $R_{n}\Vdash \{\delta \}^{\mathcal{G}}(\underline{m})\neq \underline{n}$
(for $n=0,1$) contradicting one of the basic facts about forcing.) Now
applying the construction of the fusion lemma, with the $\Pi _{1}^{1}$
operator taking $m,Q$ to such an $R$ (and $n$) for the nodes at level $m$,
gives the desired condition which, as usual, has by construction the same $k$
and $K$ values as $P$.
\end{proof}

\begin{proposition}[Diagonalization]
\label{diag}For any $x\nleq y$ in $\mathcal{K}$, $\delta <\omega _{1}^{CK}$
and $P$ in our generic filter with $x,y\in K(P)$ there is a $Q\leq P$ such
that $Q\Vdash \{\delta _{y}^{P}\}^{G}\neq G_{x}^{P}$.
\end{proposition}

\begin{proof}
Begin by taking $Q\leq P$ as in Lemma \ref{tot} for $\delta _{y}^{P}$. We
then choose any $\alpha _{0},\alpha _{1}\in \Theta _{k(P)}$ such that $%
\alpha _{0}\equiv _{y}\alpha _{1}$ but $\alpha _{0}\not\equiv _{x}\alpha
_{1} $. Such $\alpha _{0}$ and $\alpha _{1}$ exist by the differentiation
property of usl representations. Suppose $Tr(Q)(\emptyset )=Tr(P)(\sigma )$
and so $Tr(Q)(\alpha _{i})\supseteq Tr(P)(\sigma \symbol{94}\alpha _{i})$
for $i\in \{0,1\}$. Now clearly any condition $Q_{i}\leq Q_{\alpha _{i}}$
forces $G_{x}^{P}(|\sigma |)=\alpha _{i}(x)$. Let $\beta _{i}=(\alpha
_{i})^{|\sigma |},$ i.e.\ the concatenation of $|\sigma |$ many copies of $%
\alpha _{i}$. Consider then the conditions $Q_{\beta _{i}}$. They still
force $G_{x}^{P}(|\sigma |)=\alpha _{i}(x)$. On the other hand by Lemma \ref%
{tot} they each force a value for $\{\delta _{y}^{P}\}^{G}$ at $|\sigma |$.

As the $\beta _{i}$, for $i=0,1$, are congruent modulo $y$ and $y\in K(P)$,
the initial segments of $G_{y}^{P}$ that $Q_{\beta _{i}}(\emptyset )$
determine are congruent modulo $y$ as well. If $G$ is any path through $%
Q_{\beta _{0}}$ then (by uniformity of $Tr(Q)$) changing its initial segment 
$Q_{\beta _{0}}(\emptyset )$ to $Q_{\beta _{1}}(\emptyset )$ produces a path 
$G^{\prime }$ through $Q_{\beta _{1}}$ and so a corresponding one through $%
Tr(P)$ that is congruent to $G^{P}$ modulo $y$. As the value of $\{\delta
\}^{G_{y}^{P}}=\{\delta _{y}^{P}\}^{G}$ is determined by $G_{y}^{P}$ the two
conditions must force the same value for $\{\delta \}^{G_{y}^{P}}(|\sigma |)$%
. This value must be different from one of $\alpha _{i}(x)$ $=0,1$. Thus one
of the $Q_{\beta _{i}}$ forces $\{\delta \}^{G_{y}^{P}}=\{\delta
_{y}^{P}\}^{G}\neq G_{x}^{P}$ as required.
\end{proof}

We turn now to the requirement that the image of $\mathcal{K}$ under our
embedding form an initial segment of $\mathcal{D}_{h}$. This argument is
somewhat more complicated than those above and uses both the meet and
homogeneity interpolants. Still, given Lemma \ref{tot}, we present an
argument that would work for $\mathcal{D}_{T}$ as well. In that setting it
would be simpler than the existing proofs in the literature because of our
changes to the definition of the homogeneity interpolants.

We begin with the notion of a $\delta $-splitting and a lemma about such
splittings.

\begin{definition}
\label{defsplit}Given a $P=\langle T,\mathcal{\hat{K}}\rangle $ in our
generic filter with $w\in \mathcal{\hat{K}}$, a reduction $\delta $, and a
condition $Q\leq P$ and function $q$ both as in Proposition \ref{tot}, we
say that $\sigma $ and $\tau $ (of the same length) are a $\delta $%
-splitting (or $\delta $-split) on $Q$ (modulo $w$) if ($\sigma \equiv
_{w}\tau $ and) there is an $n\leq |\sigma |$ such that $q(n,\sigma
\upharpoonright n)\neq q(n,\tau \upharpoonright n)$. (So for any paths $%
G_{0} $ and $G_{1}$ on $Q$ extending $Tr(Q)(\sigma )$ and $Tr(Q)(\tau )$,
respectively, $\{\delta \}^{G_{0}}(n)\neq \{\delta \}^{G_{1}}(n)$.) If $%
R\leq Q,R(\mu )=Q(\sigma ),R(\nu )=Q(\tau )$ and $\sigma $ and $\tau $ $%
\delta $-split (modulo $w)$ on $Q$ then we also say that $\mu $ and $\nu $ $%
\delta $-split on $R$ (modulo $w$).
\end{definition}

\begin{lemma}
\label{meetsplit}Given $P$ and $Q$ as in Lemma \ref{tot}, there is a $\rho $
such that the set $Sp(\rho )=\{w\in K(P)|$ there are no $\sigma ,\tau $ that 
$\delta $-split on $Tr(Q)_{\rho }$ modulo $w\}$ is maximal. Moreover, this
maximal set is closed under meet and so has a least element $z$.
\end{lemma}

\begin{proof}
Let $T=Tr(Q)$ and $k=k(Q)=k(P)$ and $\mathcal{\hat{K}}=K(P)=K(Q)$. As $%
\mathcal{\hat{K}}$ is finite there is clearly a $\rho $ such that $Sp(\rho )$
is maximal. Note that then $Sp(\mu )=Sp(\rho )$ for any $\mu \supseteq \rho $
as $Tr(Q)_{\mu }\subseteq _{\mathcal{\hat{K}}}Tr(Q)_{\rho }$. Consider any $%
x,y\in Sp(\rho )$ with $x\wedge y=w$. As $\mathcal{\hat{K}}\subseteq _{lsl}%
\mathcal{L}$, $w\in \mathcal{\hat{K}}$. To show that $Sp(\rho )$ is closed
under meet it suffices (by the maximality of $Sp(\rho )$) to show that there
is no $\delta $-splitting on $Q_{\rho \symbol{94}0}$ modulo $z$. Let $\hat{k}
$ denote $k(Q_{\rho \symbol{94}0})=k(Q)+|\rho |+1$. Suppose there were such
a split $\bar{\mu}$ and $\bar{\nu}$. By our definition of $Q_{\rho \symbol{94%
}0}$, $\bar{\mu},\bar{\nu}\in \prod\limits_{n=0}^{n=m}\Theta _{\hat{k}+n}$
for some $m$. The splitting in $Q_{\rho }$ at the corresponding levels,
however, have branchings for all elements of $\Theta _{\hat{k}+n+1}$. Thus,
by the existence of meet interpolants for $\Theta _{\hat{k}+n}$ in $\Theta _{%
\hat{k}+n+1}$, there are $\bar{\gamma}_{0},\bar{\gamma}_{1},\bar{\gamma}%
_{2}\in \prod\limits_{n=0}^{n=m}\Theta _{\hat{k}+n+1}$ such that for each $%
j\leq m$, the $\bar{\gamma}_{i}(j)$ for $i\in \{0,1,2\}$ are meet
interpolants for $\bar{\mu}(j)$ and $\bar{\nu}(j)$, i.e.\ $\bar{\mu}\equiv
_{x}\bar{\gamma}_{0}\equiv _{y}\bar{\gamma}_{1}\equiv _{x}\bar{\gamma}%
_{2}\equiv _{x}\bar{\nu}$. As $\bar{\mu}$ and $\bar{\nu}$ form a $\delta $%
-splitting on $Q_{\rho \symbol{94}0}$ so do one of the successive pairs such
as $\bar{\gamma}_{0},\bar{\gamma}_{1}$. But then $0\symbol{94}\bar{\gamma}%
_{0}$ and $0\symbol{94}\bar{\gamma}_{1}$ would be a $\delta $-split on $%
Q_{\rho }$ congruent modulo $y$ for a contradiction. (The situations for the
other pairs are the same but perhaps with $x$ in place of $y$.)
\end{proof}

We now build the analog of what is often called a $\delta $\emph{-splitting
tree} in the Turing degree setting. It is in the construction of these trees
that our new definition of homogeneity interpolants simplifies the
construction as compared, for example, to that of Lerman [1983, VII.3].

\begin{proposition}[Splitting trees]
\label{splittree}Given $\delta ,z,P,Q$ and $\rho $ as in Lemma \ref%
{meetsplit}, there is a condition $\langle S,\mathcal{\hat{K}}\rangle \leq
Q_{\rho }$ (with $k(S)=k(Q)=k(P)=k$) such that for any $\sigma \in \text{dom }%
S(=\text{dom }Q=\text{dom }P)$ and any $\alpha ,\beta $ in the appropriate $%
\Theta _{j}$ (actually $j=k+|\sigma |$), if $\alpha \not\equiv _{z}\beta $
then $\sigma \symbol{94}\alpha $ and $\sigma \symbol{94}\beta $ $\delta $%
-split on $S$.
\end{proposition}

\begin{proof}
We define $S(\sigma )$ (with $k(S)=k$) by induction on $|\sigma |$
beginning, of course, with $S(\emptyset )=Q_{\rho }(\emptyset )$. Suppose we
have defined $S(\sigma )=Q(\tau _{\sigma })$ for all $\sigma $ of length $n$%
. We must define $S(\sigma \symbol{94}\alpha )$ for all such $\sigma $ and
appropriate $\alpha $ as extensions $Q(\tau _{\sigma \symbol{94}\alpha })$
of $Q(\tau _{\sigma }\symbol{94}\alpha )$ obeying all the congruences in $%
\mathcal{\hat{K}}$, i.e.\ if $x\in \mathcal{\hat{K}}$ and $\alpha \equiv
_{x}\beta $ then $\tau _{\sigma \symbol{94}\alpha }\equiv _{x}\tau _{\sigma 
\symbol{94}\beta }$. We list the $\sigma $ of length $n+1$ as $\sigma _{i}%
\symbol{94}\alpha _{i}$ for $i<m=|\prod\limits_{j=0}^{j=n}\Theta _{k+j}|$
and define by induction on $r$ strings $\rho _{i,r}$ for $i<m,$ $%
r<l=m(m+1)/2 $ (the number of pairs $\{i,j\}$ with $i,j<m$). At the end of
our induction we will set $\tau _{\sigma _{i}\symbol{94}\alpha _{i}}=\tau
_{\sigma _{i}}\symbol{94}\rho _{i,0}\symbol{94}\ldots \symbol{94}\rho
_{i,l-1}$. For this to succeed it suffices to guarantee for every $i,j<m$
and $w\in \mathcal{\hat{K}}$ that $\alpha _{i}\equiv _{w}\alpha
_{j}\Rightarrow \rho _{i,r}\equiv _{w}\rho _{j,r}$ for every $r<l$ and that
if $\alpha _{i}\not\equiv _{z}\alpha _{j}$ then $\tau _{\sigma _{i}}\symbol{%
94}\rho _{i,0}\symbol{94}\ldots \symbol{94}\rho _{i,r}$ and $\tau _{\sigma
_{j}}\symbol{94}\rho _{j,0}\symbol{94}\ldots \symbol{94}\rho _{r}$ $\delta $%
-split on $Q$ where $r<l$ is (the code for) $\{i,j\}$.

By induction on $r<l$ we suppose we have $\tau _{\sigma _{i}}\symbol{94}\rho
_{i,0}\symbol{94}\ldots \symbol{94}\rho _{i,r-1}=\nu _{i}$ for all $i<m$ and
that $\{p,q\}$ is pair number $r$. If $\alpha _{p}\equiv _{z}\alpha _{q}$
there is no requirement to satisfy and we let $\rho _{i,r}=\emptyset $ for
every $i$. Otherwise, let $w$ be the largest $y\in \mathcal{L}_{k+n}$ such
that $\alpha _{p}\equiv _{y}\alpha _{q}$. (To see that there is a largest
such $y$, first note that $\mathcal{L}_{k+n}$ is a lattice as it is a finite
lsl. As $\Theta _{k+n}$ is an usl representation for $\mathcal{L}_{k+n}$, if 
$\alpha _{p}\equiv _{u,v}\alpha _{q}$ for $u,v\in \mathcal{L}_{k+n}$ then $%
\alpha _{p}\equiv _{t}\alpha _{q}$ where $t$ is the least element of $%
\mathcal{L}_{k+n}$ above both $u$ and $v$ (their join from the viewpoint of $%
\mathcal{L}_{k+n}$). Thus, there is a largest $y\ $as desired.) Of course, $%
z\nleq w$. By our choice of $z$ there are $\sigma ,\tau \in
\prod\limits_{n=0}^{n=s}\Theta _{k+n}$ such that $\nu _{p}$ extended by $%
\sigma $ and $\tau $ form a $\delta $-splitting congruent modulo $w$. (We
can find such a split on $Q_{\nu _{p}}$ by definition of $\rho $ and $z$ and
our assumption on $w$. It translates into such $\sigma $ and $\tau $.)
Consider $\nu _{q}\symbol{94}\tau $. It must form a $\delta $-splitting on $%
Q $ with one of $\nu _{p}\symbol{94}\sigma $ and $\nu _{p}\symbol{94}\tau $
by the basic properties of $Q$. If it splits with the latter string then we
can set $\rho _{i,r+1}=\tau $ and clearly fulfill the requirements for this
pair $\{p,q\}$ both for congruence modulo $w$ (as all new extensions are
identical) and $\delta $-splitting. Thus we assume that $\nu _{p}\symbol{94}%
\sigma $ and $\nu _{q}\symbol{94}\tau $ $\delta $-split on $Q$. We now use
our homogeneity interpolants.

We know that $w$ is the largest $y\in \mathcal{L}_{n+k}$ such that $\alpha
_{p}\equiv _{y}\alpha _{q}$ and that $\sigma \equiv _{w}\tau $. Thus for any 
$z\in \mathcal{\hat{K}}\subseteq \mathcal{L}_{k+n}$ if $\alpha _{p}\equiv
_{z}\alpha _{q}$ then $z\leq w$ and so $\sigma \equiv _{z}\tau $. By Theorem %
\ref{repthm}(3) we can now find homogeneity interpolants $\gamma
_{0}(s),\gamma _{1}(s)$ in $\Theta _{k+s+1}$ and associated $\mathcal{\hat{K}%
}$-homomorphisms $f_{s}:\alpha _{p},\alpha _{q}\mapsto \sigma (s),\gamma
_{1}(s)$, $g_{s}:\alpha _{p},\alpha _{q}\mapsto \gamma _{0}(s),\gamma
_{1}(s) $ and $h_{s}:\alpha _{p},\alpha _{q}\mapsto \gamma _{0}(s),\tau (s)$
for each $s<|\sigma |=|\tau |$. (We let $\alpha _{0}=\alpha _{p}$, $\alpha
_{1}=\alpha _{q}$, $\beta _{0}=\sigma (s)$, $\beta _{1}=\tau (s)$, $\mathcal{%
\hat{L}}=\mathcal{\hat{K}}$ and $i=k+n$ in the Theorem.) As $\nu _{p}\symbol{%
94}\sigma $ and $\nu _{q}\symbol{94}\tau $ $\delta $-split on $Q$ one of the
pairs $\nu _{p}\symbol{94}\sigma ,\nu _{q}\symbol{94}\bar{\gamma}_{1}$; $\nu
_{p}\symbol{94}\bar{\gamma}_{0},\nu _{q}\symbol{94}\bar{\gamma}_{1}$ and $%
\nu _{p}\symbol{94}\bar{\gamma}_{0},\nu _{q}\symbol{94}\tau $ must also $%
\delta $-split on $Q$. Suppose for the sake of definiteness it is the second
pair $\nu _{p}\symbol{94}\bar{\gamma}_{0},\nu _{q}\symbol{94}\bar{\gamma}%
_{1} $. In this case we let $\rho _{i,r+1}(s)=g_{s}(\alpha _{i})$ for every $%
i$ and $s$. We use $f_{s}$ or $h_{s}$ in place of $g_{s}$ if the $\delta $%
-splitting pairs are $\nu _{p}\symbol{94}\sigma ,\nu _{q}\symbol{94}\bar{%
\gamma}_{1}$ or $\nu _{p}\symbol{94}\bar{\gamma}_{0},\nu _{q}\symbol{94}\tau 
$, respectively. By the homomorphism properties of the interpolants these
extensions preserve all the congruences in $\mathcal{\hat{K}}$ between any $%
\alpha _{i}$ and $\alpha _{j}$ as required to complete the induction and our
construction of $\delta $-splitting trees .
\end{proof}

We now conclude the proof that $\{Q|Q\Vdash \exists x\in \mathcal{K}(Q\Vdash
\{\delta \}^{G}\equiv _{h}G_{x}^{Q})\}$ is dense below $P$ for each $\delta $
and $P\in \mathcal{G}$. This will show that our embedding maps onto an
initial segment of $\mathcal{D}_{h}$ by proving that the $\delta $-splitting
tree $S$ forces that $\{\delta \}^{G}\equiv _{h}G_{z}^{S}$ (for the $z$ of
the Lemma). The required fact is generally called the Computation Lemma. We
supply the standard proof.

\begin{lemma}[Computation Lemma]
\label{comp}Given $\delta ,z,P,Q,\rho $ and $S$ as in Lemma \ref{splittree}, 
$S\Vdash \{\delta \}^{G}\equiv _{h}G_{z}^{S}$.
\end{lemma}

\begin{proof}
Let $G\in \lbrack S]$. We first show that $\{\delta \}^{G}\leq _{h}G_{z}^{S}$%
. Consider any $n$. Using $G_{z}^{S}$ we can find all the $\sigma \in \text{%
dom}S$ of length $n$ such that $\sigma (l)=G_{z}^{S}(l)$ for every $l\leq n$%
. All of these $\sigma $ are congruent modulo $z$ and so all $S_{\sigma }$
force the same value for $\{\delta \}^{G}$ at $n$. As $S(\sigma )$ is an
initial segment of $G$ for one of these $\sigma $, this value must be $%
\{\delta \}^{G}(n)$. We next argue that $G_{z}^{S}\leq _{h}\{\delta \}^{G}$.
Consider all $\sigma ,\tau \in \text{dom }S$ of length $n$. If $\sigma
\not\equiv _{z}\tau $ then, by the construction of $S$, $S_{\sigma }$ and $%
S_{\tau }$ force different values for $\{\delta \}^{G}$ at some $l<n$. Thus
using $\{\delta \}^{G}\upharpoonright n$ we can find the unique congruence
class modulo $z$ consisting of those $\sigma $ such that $S(\sigma )$ is not
ruled out as a possible initial segment of $G$. For one $\sigma $ in this
class, $S(\sigma )$ is an initial segment of $G$ and as all the $\sigma $ in
this class are congruent modulo $z$, they all determine the same values of $%
G_{z}^{S}\upharpoonright n$ which must then be the correct value.
\end{proof}

We have now completed the proof that any generic filter $\mathcal{G}$
(deciding all sentences and meeting the dense sets provided by Lemma \ref%
{tot} and Propositions \ref{diag} and \ref{splittree}) provides an embedding
of $\mathcal{K}$ onto an initial segment of $\mathcal{D}_{h}$ that sends $%
x\in \mathcal{K}$ to $\deg _{h}(G_{x}^{P})$ (for any $P\in \mathcal{G}$).

Forcing ranked sentences is a $\Pi _{1}^{1}$ relation (and so recursive in $%
\mathcal{O}$) and the inductive definition thereafter proceeds by the
standard definition quantifying over previous levels of the forcing
relation. (Remember that existential quantification is reduced to
quantification over ordinals and terms below $\omega _{1}^{CK}$ and so to
quantification over $\mathcal{O}$, while negation requires quantification
over forcing conditions and so over $\mathcal{O}\oplus \mathcal{K}$.) Thus
the forcing relation for arbitrary sentences is arithmetic in $\mathcal{O}%
\oplus \mathcal{K}$ (at a level simply related to the complexity of the
sentence). Thus one can, as usual, construct a generic sequence and
corresponding filter that decides all sentences hyperarithmetically in $%
\mathcal{O}\oplus \mathcal{K}$. The steps for meeting the requirements of
totality, diagonalization and splitting are also obviously computable in the
forcing relation once one has proven the existence of the desired extensions
(Lemma \ref{tot}, Proposition \ref{diag} and \ref{splittree}). Thus we can
find our generic $G\equiv _{h}G_{1}^{Id}$ (remember, $1$ is the top element
of our lattice), such that the hyperdegrees below that of $G$ are isomorphic
to our given $\mathcal{K}$, hyperarithmetically in $\mathcal{O}\oplus 
\mathcal{K}$. This establishes Theorem \ref{main} given our lattice
representation theorem to whose proof we now turn.

\section{The lattice representation theorem\label{repsec}}

\begin{theorem}
\label{repthm}If $\mathcal{L}$ is a countable lattice then there is an usl
representation $\Theta $ of $\mathcal{L}$ along with a nested sequence of
finite slsls $\mathcal{L}_{i}$ starting with\emph{\ }$\mathcal{L}%
_{0}=\{0,1\} $ with union $\mathcal{L}$ and a nested sequence of finite
subsets $\Theta _{i}$ with union $\Theta $ with both sequences recursive in $%
\mathcal{L}$ with the following properties:

\begin{enumerate}
\item For each $i,$ $\Theta _{i}\upharpoonright \mathcal{L}_{i}$ is an usl
representation of $\mathcal{L}_{i}$.

\item There are \emph{meet interpolants} for $\Theta _{i}$ in $\Theta _{i+1}$%
, i.e.\ if $\alpha \equiv _{z}\beta $, $x\wedge y=z$ (in $\Theta _{i}$ and $%
\mathcal{L}_{i}$, respectively) then there are $\gamma _{0},\gamma
_{1},\gamma _{2}\in \Theta _{i+1}$ such that $\alpha \equiv _{x}\gamma
_{0}\equiv _{y}\gamma _{1}\equiv _{x}\gamma _{2}\equiv _{y}\beta $.

\item For every $\mathcal{\hat{L}}\subseteq _{lsl}\mathcal{L}_{i}$ there are 
\emph{homogeneity interpolants} for $\Theta _{i}$ with respect to $\mathcal{%
\hat{L}}$ in $\Theta _{i+1}$, i.e.\ for every $\alpha _{0},\alpha _{1},\beta
_{0},\beta _{1}\in \Theta _{i}$ such that $\forall w\in \mathcal{\hat{L}}%
(\alpha _{0}\equiv _{w}\alpha _{1}\rightarrow \beta _{0}\equiv _{w}\beta
_{1})$, there are $\gamma _{0},\gamma _{1}\in \Theta _{i+1}$ and $\mathcal{%
\hat{L}}$-homomorphisms $f,g,h:\Theta _{i}\rightarrow \Theta _{i+1}$ such
that $f:\alpha _{0},\alpha _{1}\mapsto \beta _{0},\gamma _{1}$, $g:\alpha
_{0},\alpha _{1}\mapsto \gamma _{0},\gamma _{1}$ and $h:\alpha _{0},\alpha
_{1}\mapsto \gamma _{0},\beta _{1}$.
\end{enumerate}
\end{theorem}

\begin{proof}
We first define the sequence $\mathcal{L}_{i}$ of slsls of $\mathcal{L}$
beginning with $\mathcal{L}_{0}$ which consists of the $0$ and $1$ of $%
\mathcal{L}$. We let the other elements of $\mathcal{L}$ be $x_{n}$ for $%
n\geq 1$ and $\mathcal{L}_{n}$ be the (necessarily finite) slsl of $\mathcal{%
L}$ generated by $\{0,1,x_{1},\ldots ,x_{n}\}$. As for $\Theta $, we choose
a countable set $\alpha _{i}$ and stipulate that $\Theta =\{\alpha _{i}|i\in
\omega \}$. We begin defining the (values of) the $\alpha _{i}$ by setting $%
\alpha _{0}(x)=0$ for all $x\in \mathcal{L}$ and $\alpha (0)=0$ for all $%
\alpha \in \Theta $. We will now define $\Theta _{n}$ and the values of $%
\alpha \in \Theta _{n}$ (other than $\alpha _{0}$) on the elements of $%
\mathcal{L}_{n}$ (other than $0$) by recursion. For $\Theta _{0}$ we choose
a new element $\beta $ of $\Theta $ and let $\Theta _{0}=\{\alpha _{0},\beta
\}$ and set $\beta (1)=1$. Given $\Theta _{n}$ and the values for its
elements on $\mathcal{L}_{n}$ we wish to enlarge $\Theta _{n}$ to $\Theta
_{n+1}$ and define the values of $\alpha (x)$ for $\alpha \in \Theta _{n+1}$
and $x\in \mathcal{L}_{n+1}$ so that the requirements of the Theorem are
satisfied. To do this we prove a number of general extension theorems for
usl representations in the Propositions below that show that we can make
simple extensions to satisfy any particular meet or homogeneity requirement
and also extend usl representations from smaller to larger slsls of $%
\mathcal{L}$. To be more specific, we first apply Proposition \ref%
{meetinterp} successively for each choice of $x\wedge y=z$ in $\mathcal{L}%
_{n}$ and $\alpha ,\beta \in \Theta _{n}$ with $\alpha \equiv _{z}\beta $
choosing new elements of $\Theta $ to form $\Theta _{n}^{\prime }$ extending 
$\Theta _{n}$ and defining them on $\mathcal{L}_{n}$ so that $\Theta
_{n}^{\prime }\upharpoonright \mathcal{L}_{n}$ is an usl representation for $%
\mathcal{L}_{n}$ containing $\Theta _{n}$ and the required meet interpolants
for every such $x,y,z,\alpha $ and $\beta $. We then apply Proposition \ref%
{hominterp} successively for each $\mathcal{\hat{L}}\subseteq _{lsl}\mathcal{%
L}_{n}$ and each $\alpha _{0},\alpha _{1},\beta _{0},\beta _{1}\in \Theta
_{n}$ such that $\forall w\in \mathcal{\hat{L}}(\alpha _{0}\equiv _{w}\alpha
_{1}\rightarrow \beta _{0}\equiv _{w}\beta _{1})$ to get larger subset $%
\Theta _{n}^{\prime \prime }$ of $\Theta $ which we also define on $\mathcal{%
L}_{n}$ so as to have an usl representation $\Theta _{n}^{\prime \prime
}\upharpoonright \mathcal{L}_{n}$ for $\mathcal{L}_{n}$ that has the
required homogeneity interpolants and $\mathcal{\hat{L}}$-homomorphisms from 
$\Theta _{n}$ into $\Theta _{n}^{\prime \prime }$ for every such $\alpha
_{0},\alpha _{1},\beta _{0},\beta _{1}\in \Theta _{n}$. Finally, we apply
Proposition \ref{extend} to define the elements of $\Theta _{n}^{\prime
\prime }$ on $\mathcal{L}_{n+1}$ and further enlarge it to our desired
finite $\Theta _{n+1}\subseteq \Theta $ with all its new elements also
defined on $\mathcal{L}_{n+1}$ so as to have an usl representation of $%
\mathcal{L}_{n+1}$ with all the properties required by the Theorem. It is
now easy immediate from the definitions that the union $\Theta $ of the $%
\Theta _{n}$ is an usl representation of $\mathcal{L}$.
\end{proof}

\begin{notation}
If a finite $\mathcal{\hat{L}}$ is a slsl of $\mathcal{L}$, $\mathcal{\hat{L}%
}\subseteq _{lsl}\mathcal{L}$, and $x\in \mathcal{L}$ then we let $\hat{x}$
denote the least element of $\mathcal{\hat{L}}$ above $x$. The desired
element of $\mathcal{\hat{L}}$ exists because $\mathcal{\hat{L}}$ is a slsl
of $\mathcal{L}$ and so the infimum (in $\mathcal{\hat{L}}$ or,
equivalently, in $\mathcal{L}$) of $\{u\in \mathcal{\hat{L}}|x\leq u\}$ is
in $\mathcal{\hat{L}}$ and is the desired $\hat{x}$. As $\mathcal{\hat{L}}$
is finite it is also a lattice but join in $\mathcal{\hat{L}}$ may not agree
with that in $\mathcal{L}$. We denote them by $\vee _{\mathcal{\hat{L}}}$
and $\vee _{\mathcal{L}}$ respectively when it is necessary to make this
distinction.
\end{notation}

\begin{lemma}
\label{basichat}With the notation as above, $\hat{x}=x$ for $x\in \mathcal{%
\hat{L}}$ and so it is an idempotent operation. If $x\leq y$ are in $%
\mathcal{L}$ then $\hat{x}\leq \hat{y}$. If $x\vee _{\mathcal{L}}y=z$ are in 
$\mathcal{L}$ then $\hat{z}=\hat{x}\vee _{\mathcal{\hat{L}}}\hat{y}$.
\end{lemma}

\begin{proof}
The first two assertions follow immediately from the definition of $\hat{x}$%
. The third is only slightly less immediate: $x,y\leq x\vee _{\mathcal{L}%
}y=z $ and so by the second assertion, $\hat{x},\hat{y}\leq \hat{z}$ and so $%
\hat{x}\vee _{\mathcal{\hat{L}}}\hat{y}\leq \hat{z}$. For the other
direction, note that as $x\leq \hat{x}$, $y\leq \hat{y}$, we have that $%
z=x\vee _{\mathcal{L}}y\leq \hat{x}\vee _{\mathcal{L}}\hat{y}\leq \hat{x}%
\vee _{\mathcal{\hat{L}}}\hat{y}\in \mathcal{\hat{L}}$ and so $\hat{z}\leq 
\hat{x}\vee _{\mathcal{\hat{L}}}\hat{y}$.
\end{proof}

\begin{proposition}
\label{extend}If $\Theta $ is a finite usl representation for $\mathcal{\hat{%
L}}\subseteq _{lsl}\mathcal{L}$ (finite) then there are extensions for each $%
\alpha \in \Theta $ to maps with domain $\mathcal{L}$ and finitely many
further functions $\beta $ with domain $\mathcal{L}$ such that adding them
on to our extensions of the $\alpha \in \Theta $ provides an usl
representation $\Theta ^{\prime }$ of $\mathcal{L}$ with $\Theta \subseteq
\Theta ^{\prime }\upharpoonright \mathcal{\hat{L}}$. Moreover, these
extensions can be found uniformly recursively in the given data ($\Theta $, $%
\mathcal{\hat{L}}$ and $\mathcal{L}$).
\end{proposition}

\begin{proof}
For $\alpha \in \Theta $ and $x\in \mathcal{L}$ set $\alpha (x)=\alpha (\hat{%
x})$. We first check that we have maintained the order and join properties
required of an usl representation. If $x\leq y$ are in $\mathcal{L}$, $%
\alpha ,\beta \in \Theta $ and $\alpha \equiv _{y}\beta $ then by definition 
$\alpha \equiv _{\hat{y}}\beta $ and so $\alpha \equiv _{\hat{x}}\beta $ as $%
\hat{x}\leq \hat{y}$ by Lemma \ref{basichat} and $\Theta $'s being an usl
representation of $\mathcal{\hat{L}}$. Thus, by definition, $\alpha \equiv
_{x}\beta $ as required.

Next, if $x\vee _{\mathcal{L}}y=z$ are in $\mathcal{L}$ and $\alpha \equiv
_{x,y}\beta $ we wish to show that $\alpha \equiv _{z}\beta $. Again by
definition $\alpha \equiv _{\hat{x},\hat{y}}\beta $. By Lemma \ref{basichat}%
, $\hat{x}\vee _{\mathcal{L}}\hat{y}=z$, so by $\Theta $ being an usl
representation for $\mathcal{\hat{L}}$, $\alpha \equiv _{\hat{z}}\beta $ and
so by definition, $\alpha \equiv _{z}\beta $.

All that remains is to show that we can add on new maps with domain $%
\mathcal{L}$ that provide witnesses for the differentiation property for
elements of $\mathcal{L}-\mathcal{\hat{L}}$ while preserving the order and
join properties. This is a standard construction. For each pair $x\nleq y$
(in $\mathcal{L}$ but not both in $\mathcal{\hat{L}}$) in turn we add on new
elements $\alpha _{x,y}$ and $\beta _{x,y}$ with all new and distinct values
at each $z\in \mathcal{L}$ except that they agree on all $z\leq x$ (and at $%
0 $, of course, have value $0$). These new elements obviously provide the
witnesses required for the differentiation property for an usl
representation. It is easy to see that they also cause no damage to the
order or join properties. There are no new nontrivial instances of
congruences between them and the old ones in $\Theta $ (extended to $%
\mathcal{L}$). Among the new elements the only instances to consider are
ones between $\alpha _{x,y}$ and $\beta _{x,y}$ for the same pair $x,y$ and
for lattice elements $z$ less than or equal to $x$. As $\alpha _{x,y}\equiv
_{z}\beta _{x,y}$ for all $z\leq x$, the order and join properties are
immediate.
\end{proof}

\begin{proposition}
\label{meetinterp}If $\alpha ,\beta \in \Theta $, an usl representation for
a finite lattice $\mathcal{L}$, $\alpha \equiv _{z}\beta $ and $x\wedge y=z$
in $\mathcal{L}$ then there are $\gamma _{0},\gamma _{1},\gamma _{2}$ such
that $\alpha \equiv _{x}\gamma _{0}\equiv _{y}\gamma _{1}\equiv _{x}\gamma
_{2}\equiv _{y}\beta $ and $\Theta \cup \{\gamma _{0},\gamma _{1},\gamma
_{2}\}$ is still an usl representation for $\mathcal{L}$. Moreover, these
extensions can be found uniformly recursively in the given data.
\end{proposition}

\begin{proof}
This is a standard fact going back to Jonsson [1953] and can be found in
Lerman [1983, Appendix B.2.5]. If $x\leq y$, there is nothing to be proved.
Otherwise, the interpolants can be defined by letting $\gamma _{0}(w)$ be $%
\alpha (w)$ for $w\leq x$ and new values for $w\nleq x$; $\gamma
_{1}(w)=\gamma _{0}(w)$ for $w\leq y$ and new values otherwise; and $\gamma
_{2}(w)=\beta (w)$ for $w\leq y$, $\gamma _{2}(w)=\gamma _{1}(w)$ if $w\leq
x $ but $w\nleq y$ and new otherwise.
\end{proof}

\begin{proposition}
\label{hominterp}If $\mathcal{\hat{L}}\subseteq _{lsl}\mathcal{L}$, a finite
lattice, and $\Theta $ is an usl representation for $\mathcal{L}$ with $%
\alpha _{0},\alpha _{1},\beta _{0},\beta _{1}\in \Theta $ such that $\forall
w\in \mathcal{\hat{L}}(\alpha _{0}\equiv _{w}\alpha _{1}\rightarrow \beta
_{0}\equiv _{w}\beta _{1})$, then there is an usl representation $\tilde{%
\Theta}\supseteq \Theta $ for $\mathcal{L}$ with $\gamma _{0},\gamma _{1}\in 
\tilde{\Theta}$ and $\mathcal{\hat{L}}$ homomorphisms $f,g,h:\Theta
\rightarrow \tilde{\Theta}$ such that $f:\alpha _{0},\alpha _{1}\mapsto
\beta _{0},\gamma _{1}$, $g:\alpha _{0},\alpha _{1}\mapsto \gamma
_{0},\gamma _{1}$ and $h:\alpha _{0},\alpha _{1}\mapsto \gamma _{0},\beta
_{1}$. Moreover, these extensions can be found uniformly recursively in the
given data.
\end{proposition}

\begin{proof}
For each $\alpha \in \Theta $ and $x\in \mathcal{L}$ we set $f(\alpha
)(x)=\beta _{0}(x)$ if $\alpha \equiv _{\hat{x}}\alpha _{0}$ and otherwise
we let it be a new number that depends only on $\alpha (\hat{x})$, e.g. $%
\alpha (\hat{x})^{\ast }$. Note that which case of the definition applies
for $f(\alpha )(x)$ depends only on $\alpha (\hat{x})$ and it can be an
\textquotedblleft old\textquotedblright\ value (i.e.\ one of some $\beta \in
\Theta $) only in the first case. Thus, for $\alpha ,\beta \in \Theta $,%
\begin{equation}
\text{(a) }\alpha \equiv _{\hat{x}}\beta \Leftrightarrow f(\alpha )\equiv
_{x}f(\beta )\text{ and (b) }f(\alpha )\equiv _{x}\beta \Rightarrow \alpha
\equiv _{\hat{x}}\alpha _{0}\Rightarrow f(\alpha )\equiv _{x}\beta _{0}\text{%
.}  \label{1}
\end{equation}%
Let $\Theta _{1}=\Theta \cup f[\Theta ]$. We claim that $\Theta _{1}$ is an
usl representation for $\mathcal{L}$ and $f$ is an $\mathcal{\hat{L}}$%
-homomorphism from $\Theta $ into $\Theta _{1}$. That $f$ is an $\mathcal{%
\hat{L}}$-homomorphism is immediate from the first clause in (\ref{1}) and
the fact (Lemma \ref{basichat}) that $\hat{x}=x$ for $x\in \mathcal{\hat{L}}$%
. We next check that $\Theta _{1}$ satisfies the properties required of an
usl representation. Of course, $f(\alpha )(0)=0$ by definition for every $%
\alpha $ and differentiation is automatic as it extends $\Theta $.

First, to check the order property for $\Theta _{1}$ we consider any $x\leq
y $ in $\mathcal{L}$. As $\Theta $ is already an usl representation for $%
\mathcal{L}$, it suffices to consider two cases for the pair of elements of $%
\Theta _{1}$ which are given as congruent modulo $y$ and show that in these
two cases they are also congruent modulo $x$. The two cases are that (a)
both are in $f[\Theta ]$ and that (b) one is in $f[\Theta ]$ and the other
in $\Theta $. Thus it suffices to consider any $\alpha ,\beta \in \Theta $,
assume that (a) $f(\alpha )\equiv _{y}f(\beta )$ or (b) $f(\alpha )\equiv
_{y}\beta $ and prove that (a) $f(\alpha )\equiv _{x}f(\beta )$ and (b) $%
f(\alpha )\equiv _{x}\beta $, respectively. For (a), we have by (\ref{1})
that $\alpha \equiv _{\hat{y}}\beta $ and so by the order property for $%
\Theta $, $\alpha \equiv _{\hat{x}}\beta $. Thus $f(\alpha )\equiv
_{x}f(\beta )$ by definition as required. As for (b), (\ref{1}) tells us
here that $\alpha \equiv _{\hat{y}}\alpha _{0}$ and $\beta \equiv
_{y}f(\alpha )\equiv _{y}\beta _{0}$ (and therefore $\beta \equiv _{x}\beta
_{0}$). Now by Lemma \ref{basichat} $\alpha \equiv _{\hat{x}}\alpha _{0}$ so 
$f(\alpha )\equiv _{x}\beta _{0}$ and so $f(\alpha )\equiv _{x}\beta $ as
required.

Next we verify the join property for $x\vee y=z$ in $\mathcal{L}$ and two
elements of $\Theta _{1}$ (not both in $\Theta $) in the same two cases. For
(a) we have that $f(\alpha )\equiv _{x,y}f(\beta )$ and so as above $\alpha
\equiv _{\hat{x},\hat{y}}\beta $. Now by the join property in $\Theta $ and
Lemma \ref{basichat}, $\alpha \equiv _{\hat{z}}\beta $ and so $f(\alpha
)\equiv _{z}f(\beta )$ as required. For (b) using (\ref{1}b) and Lemma \ref%
{basichat} again we have that $f(\alpha )\equiv _{x,y}\beta \Rightarrow $ $%
\alpha \equiv _{\hat{x},\hat{y}}\alpha _{0}\Rightarrow \alpha \equiv _{\hat{z%
}}\alpha _{0}\Rightarrow f(\alpha )\equiv _{z}\beta _{0}$ while it also
tells us that $\beta \equiv _{x,y}f(\alpha )\equiv _{x,y}\beta _{0}$ as
required. Note that clearly $f(\alpha _{0})=\beta _{0}$. We let $\gamma
_{1}=f(\alpha _{1})$ and so have the first function and (partial) extension
of $\Theta $ required in the Proposition.

We now define $h$ on $\Theta _{1}$ as we did $f$ on $\Theta $ using $\alpha
_{1}$ and $\beta _{1}$ in place of $\alpha _{0}$ and $\beta _{0}$,
respectively: $h(\alpha )(x)=\beta _{1}(x)$ if $\alpha \equiv _{\hat{x}%
}\alpha _{1}$ and otherwise we let it be a new number that depends only on $%
\alpha (\hat{x})$, e.g. $\alpha (\hat{x})^{\ast \ast }$. Let $\Theta
_{2}=\Theta _{1}\cup h[\Theta _{1}]$. As above, $\Theta _{2}$ is an usl
representation for $\mathcal{L}$ and $h$ is an $\mathcal{\hat{L}}$%
-homomorphism from $\Theta _{1}$ (and so $\Theta $) into $\Theta _{2}$
taking $\alpha _{1}$ to $\beta _{1}$. We let $\gamma _{0}=h(\alpha _{0})$
and so have the third function and (partial) extension of $\Theta $ required
in the Proposition. As above in (\ref{1}), we have for any $\alpha ,\beta
\in \Theta _{1}$ and $x\in \mathcal{L}$,%
\begin{equation}
\text{(a) }\alpha \equiv _{\hat{x}}\beta \Leftrightarrow h(\alpha )\equiv
_{x}h(\beta )\text{ and (b) }h(\alpha )\equiv _{x}\beta \Rightarrow \alpha
\equiv _{\hat{x}}\alpha _{1}\Rightarrow h(\alpha )\equiv _{x}\beta _{1}\text{%
.}  \label{2}
\end{equation}%
Applying the second clause to $\gamma _{0}=h(\alpha _{0})$ and first to any $%
\beta \in \Theta _{1}$ and then, in particular to $\gamma _{1}$ we have%
\begin{equation}
\text{(a) }\gamma _{0}\equiv _{x}\beta \Rightarrow \alpha _{0}\equiv _{\hat{x%
}}\alpha _{1}\Rightarrow f(\alpha _{1})=\gamma _{1}\equiv _{x}\beta _{0}%
\text{ and (b) }\gamma _{0}\equiv _{x}\gamma _{1}\Leftrightarrow \alpha
_{0}\equiv _{\hat{x}}\alpha _{1}\text{.}  \label{3}
\end{equation}%
\noindent To see the right to left direction of the second clause, note that 
$\alpha _{0}\equiv _{\hat{x}}\alpha _{1}$ implies that $\gamma _{0}\equiv
_{x}\beta _{1}$ and $\gamma _{1}\equiv _{x}\beta _{0}$ by the definitions of 
$h$ and $f$, respectively, while it also implies that $\beta _{0}\equiv _{%
\hat{x}}\beta _{1}$ by the basic assumption of the Proposition. Thus, as $%
\Theta $ is an usl representation of $\mathcal{L}$ and $x\leq \hat{x}$, $%
\beta _{0}\equiv _{x}\beta _{1}$ and $\gamma _{0}\equiv _{x}\gamma _{1}$.

Finally, we define $g$ on $\alpha \in \Theta _{2}$ by setting $g(\alpha
)(x)=\gamma _{0}(x)$ if $\alpha \equiv _{\hat{x}}\alpha _{0}$. If $\alpha
\not\equiv _{\hat{x}}\alpha _{0}$ but $\alpha \equiv _{\hat{x}}\alpha _{1}$
then $g(\alpha )(x)=\gamma _{1}(x)$. Otherwise, we let $g(\alpha )(x)$ be a
new number that depends only on $\alpha (\hat{x})$, e.g. $\alpha (\hat{x}%
)^{\ast \ast \ast }$. Note that if $\alpha \equiv _{\hat{x}}\alpha _{1}$
then we always have $g(\alpha )\equiv _{x}\gamma _{1}$ as if $\alpha \equiv
_{\hat{x}}\alpha _{0}$ as well then, by (\ref{3}b), $\gamma _{0}\equiv
_{x}\gamma _{1}$. Thus $g(\alpha _{0})=\gamma _{0}$ and $g(\alpha
_{1})=\gamma _{1}$ as required. It is also obvious that $g$ is an $\mathcal{%
\hat{L}}$-homomorphism of $\Theta _{2}$ (and so $\Theta $) into $\Theta
_{3}=\Theta _{2}\cup g[\Theta _{2}]$ as by definition and Lemma \ref%
{basichat}, $\alpha \equiv _{\hat{x}}\beta \Rightarrow g(\alpha )\equiv _{%
\hat{x}}g(\beta )$ for any $x\in \mathcal{L}$. Indeed, for any $\alpha
,\beta \in \Theta _{2}$ and $x\in \mathcal{L}$%
\begin{equation}
\alpha \equiv _{\hat{x}}\beta \Leftrightarrow g(\alpha )\equiv _{x}g(\beta )%
\text{.}  \label{4}
\end{equation}%
\noindent To see the right to left direction here, note that if either of $%
g(\alpha )$ or $g(\beta )$ is new for $g$ at $x$ (i.e.\ of the form $\delta (%
\hat{y})^{\ast \ast \ast }$) then clearly both are. In this case, $\alpha
\equiv _{\hat{x}}\beta $ by definition. Otherwise, either they are both
congruent to $\alpha _{0}$ or both to $\alpha _{1}$ and so congruent to each
other mod $\hat{x}$. The point here is that if one is congruent to $\alpha
_{0}$ and the other to $\alpha _{1}$ but not $\alpha _{0}$ at $\hat{x}$ then
by definition $\gamma _{0}\equiv _{x}\gamma _{1}$ and so by (\ref{3}b), $%
\alpha _{0}\equiv _{\hat{x}}\alpha _{1}$ for a contradiction.

Thus we only need to verify that $\Theta _{3}$ is an usl representation of $%
\mathcal{L}$. We consider any $\alpha ,\beta \in \Theta _{2}$ and divide the
verifications into cases (a) and (b) as before with the former considering $%
g(\alpha )$ and $g(\beta )$ and the latter $g(\alpha )$ and $\beta $. These
cases may then be further subdivided.

We begin with the order property and so $x\leq y$ in $\mathcal{L}$.

(a) If $g(\alpha )\equiv _{y}g(\beta )$ then, by (\ref{4}), $\alpha \equiv _{%
\hat{y}}\beta $ and so $\alpha \equiv _{\hat{x}}\beta $ as $\hat{x}\leq \hat{%
y}$ (Lemma \ref{basichat}) and $\Theta _{2}$ is an usl representation of $%
\mathcal{L}$. Thus, again by (\ref{4}) $g(\alpha )\equiv _{x}g(\beta )$ as
required.

(b) If $g(\alpha )\equiv _{y}\beta $ then by definition they are congruent
modulo $y$ to $\gamma _{i}$ (for some $i\in \{0,1\}$) and $\alpha $ is
congruent to $\alpha _{i}$ at $\hat{y}$. Thus $\alpha \equiv _{\hat{x}%
}\alpha _{i}$ as $\hat{x}\leq \hat{y}$ and $\Theta _{2}$ is an usl
representation so $g(\alpha )\equiv _{x}\gamma _{i}$ by definition.
Similarly, as $x\leq y$, $\beta \equiv _{x}\gamma _{i}$ as well.

Now for the join property for $x\vee y=z$ in $\mathcal{L}$.

(a) If $g(\alpha )\equiv _{x,y}g(\beta )$ then, as above, $\alpha \equiv _{%
\hat{x},\hat{y}}\beta $. As $\hat{x}\vee \hat{y}=\hat{z}$ by Lemma \ref%
{basichat} and $\Theta _{2}$ is an usl representation, $\alpha \equiv _{\hat{%
z}}\beta $ and so by (\ref{4}) $g(\alpha )\equiv _{z}g(\beta )$ as required.

(b) If $g(\alpha )\equiv _{x,y}\beta $ then again $\alpha \equiv _{\hat{x}%
}\alpha _{i}$ and $\alpha \equiv _{\hat{y}}\alpha _{j}$ for some $i,j\in
\{0,1\}$ and $g(\alpha )\equiv _{x}\beta \equiv _{x}\gamma _{i}$ while $%
g(\alpha )\equiv _{y}\beta \equiv _{y}\gamma _{j}$. If $i=j$ then $\alpha
\equiv _{\hat{x},\hat{y}}\alpha _{i}$ and so $\alpha \equiv _{\hat{z}}\alpha
_{i}$ and $g(\alpha )\equiv _{z}\gamma _{i}\equiv _{z}\beta $ as required.

On the other hand, suppose (without loss of generality) that $\alpha \equiv
_{\hat{x}}\alpha _{0}$ and so $\beta \equiv _{x}g(\alpha )\equiv _{\hat{x}%
,x}\gamma _{0}=h(\alpha _{0})$ while $\alpha _{0}\not\equiv _{\hat{y}}\alpha
\equiv _{\hat{y}}\alpha _{1}$ and so $\beta \equiv _{y}g(\alpha )\equiv _{%
\hat{y},y}\gamma _{1}=f(\alpha _{1})$. If $\beta \in \Theta _{1}$ then by (%
\ref{4}a) $\alpha _{0}\equiv _{\hat{x}}\alpha _{1}$ and so $\alpha \equiv _{%
\hat{x}}\alpha _{1}$. As our assumption is that $\alpha \equiv _{\hat{y}%
}\alpha _{1}$ we have (by the join property in $\Theta _{2}$) that $\alpha
\equiv _{\hat{z}}\alpha _{1}$ and so $g(\alpha )\equiv _{z}\gamma _{1}$. As $%
\alpha _{0}\equiv _{\hat{x}}\alpha _{1}$ (\ref{3}b) tells us that $\gamma
_{0}\equiv _{x}\gamma _{1}$. Our assumptions then say that $\beta \equiv
_{x,y}\gamma _{1}$ and so $\beta \equiv _{z}\gamma _{1}$ as required. Thus
we may assume that $\beta =h(\delta )$ for some $\delta \in \Theta _{1}$.

We now have $h(\delta )=\beta \equiv _{x}g(\alpha )\equiv _{x}\gamma
_{0}=h(\alpha _{0})\in \Theta _{1}$ and so by (\ref{2}a) applied to $%
h(\delta )\equiv _{x}h(\alpha _{0})$ with $\delta $ for $\alpha $ and $%
\alpha _{0}$ for $\beta $ we see that $\delta \equiv _{\hat{x}}\alpha _{0}$.
We also have $h(\delta )=\beta \equiv _{y}g(\alpha )\equiv _{\hat{y}%
,y}\gamma _{1}=f(\alpha _{1})$. Applying (\ref{2}b) to $h(\delta )\equiv
_{y}\gamma _{1}$ with $\delta $ for $\alpha $ and $\gamma _{1}\in \Theta
_{1} $ for $\beta $, we see that $\delta \equiv _{\hat{y}}\alpha _{1}$and $%
h(\delta )\equiv _{y}\beta _{1}$ and so $\beta _{1}\equiv _{y}\gamma
_{1}=f(\alpha _{1})$. Now applying (\ref{1}b) with $\alpha _{1}$ for $\alpha 
$ and $\beta _{1}\in \Theta $ for $\beta $, we have that $\alpha _{1}\equiv
_{\hat{y}}\alpha _{0}$. As this contradicts our (without loss of generality)
assumption, we are done.
\end{proof}

We complete our lattice theoretic material by supplying a proof of Lemma \ref%
{amal}.

\noindent \textbf{Lemma \ref{amal}: }\emph{The class of finite lattices
(with }$0$ \emph{and} $1$\emph{) has the amalgamation property, i.e.\ if }$%
\mathcal{A}$, $\mathcal{B}_{0}$ \emph{and} $\mathcal{B}_{1}$ \emph{are
finite lattices and} $f_{0},f_{1}$ \emph{are embeddings of} $\mathcal{A}$ 
\emph{into} $\mathcal{B}_{0}$ \emph{and} $\mathcal{B}_{1}$\emph{,
respectively, then there is a finite lattice }$\mathcal{C}$ \emph{and
embeddings} $g_{0}$ \emph{and} $g_{1}$ \emph{of} $\mathcal{B}_{0}$ \emph{and}
$\mathcal{B}_{1}$\emph{, respectively, into} $\mathcal{C}$ \emph{such that} $%
g_{0}f_{0}\upharpoonright \mathcal{A}=g_{1}f_{1}\upharpoonright \mathcal{A}$.

\begin{proof}
To simplify the notation we assume without loss of generality that the
embeddings $f_{0}$ and $f_{1}$ are the inclusion maps and that the elements
of $\mathcal{A}$ are the only ones that $\mathcal{B}_{0}$ and $\mathcal{B}%
_{1}$ have in common. We begin with a partial lattice $\mathcal{P}$ whose
universe is the union of those of $\mathcal{B}_{0}$ and $\mathcal{B}_{1}$.
We define an order on $\mathcal{P}$ that coincides with the one on $\mathcal{%
B}_{i}$ for $x,y$ both in one $\mathcal{B}_{i}$ and otherwise (say $x\in 
\mathcal{B}_{i}-\mathcal{A}$ and $y\in \mathcal{B}_{1-i}-\mathcal{A}$) we
set $x<y\Leftrightarrow \exists a\in \mathcal{A}(x<_{\mathcal{B}%
_{i}}a~\&~a<_{\mathcal{B}_{1-i}}y)$. This relation is clearly transitive and
we claim it preserves both join and meet from each $\mathcal{B}_{i}$. First,
if $x,y\in \mathcal{B}_{i}$ and $x\wedge y=z$ in $\mathcal{B}_{i}$ then $z$
is also the greatest lower bound of $x$ and $y$ in $\mathcal{P}$. Clearly $%
z\leq x,y$. So suppose $w\leq x,y$ is in $\mathcal{P}$. If $w\in \mathcal{B}%
_{i}$ then, of course, $w\leq z$. If $w\in \mathcal{P}-\mathcal{B}_{i}$ then
there are $a_{x},a_{y}\in \mathcal{A}$ such that $w<a_{x}\leq x$ and $%
w<a_{y}\leq y$ so $w\leq a_{x}\wedge _{\mathcal{A}}a_{y}\leq a_{x}\wedge _{%
\mathcal{B}_{i}}a_{y}\leq x\wedge _{\mathcal{B}_{i}}y=z$. The argument for
preserving join is similar.

Now let $\mathcal{C}$ be the set of ideals of $\mathcal{P}$, i.e.\ the
subsets of $\mathcal{P}$ closed downward and under join when defined in $%
\mathcal{P}$. We first note that $\mathcal{C}$ is clearly a lattice with
order given by containment and the operations on $X,Y\in \mathcal{C}$ given
by $X\wedge Y=X\cap Y$ and $X\vee Y$ equals the ideal in $\mathcal{P}$
generated by $X\cup Y$ (i.e.\ we close downward and under join when defined).

Finally, we define the required maps $g_{i}:\mathcal{B}_{i}\rightarrow 
\mathcal{C}$ as the restrictions (to $\mathcal{B}_{i}$) of a single one-one $%
g:\mathcal{P}\rightarrow \mathcal{C}$ defined by sending $p\in \mathcal{P}$
to $\{q\in \mathcal{P}|q\leq _{\mathcal{P}}p\}$ (the ideal generated by $p$%
). We show that $g$ preserves join and meet in $\mathcal{P}$ when they exist
and so its restrictions to $\mathcal{B}_{i}$ are lattice embeddings. If $%
p\wedge q=r$ in $\mathcal{P}$ then it is clear that $g(p)\cap g(q)=g(r)$ by
the definition of meet in $\mathcal{P}$ as required. As for join, if $p\vee
q=r$ in $\mathcal{P}$ then $g(r)$ is an ideal of $\mathcal{P}$ that contains
both $g(p)$ and $g(q)$. On the other hand, any ideal of $\mathcal{P}$
containing both $p$ and $q$ must contain $r$ by the definition of ideals in $%
\mathcal{P}$. Thus $g(r)=g(p)\vee _{\mathcal{C}}g(q)$ as required.
\end{proof}

\section{Usls and other questions\label{quest}}

There are now two obvious questions about the possible countable initial
segments of the hyperdegrees. The first asks about lattice initial segments.

\begin{question}
What are the lattice initial segments of $\mathcal{D}_{h}$? In particular,
are there any which are not sublattices of some hyperarithmetic lattice?
\end{question}

We do not even have any candidates for additional lattices isomorphic to
initial segments of $\mathcal{D}_{h}$.

The second natural line of inquiry asks about usl initial segments. In $%
\mathcal{D}_{T}$, there is no difference in the results: Every countable usl
is isomorphic to an initial segment of $\mathcal{D}_{T}$. Of course, we have
seen (Theorem \ref{ctrex}) that not every countable lattice is isomorphic to
an initial segment of $\mathcal{D}_{h}$. Given our Theorem \ref{main},
however, the conjecture might be that every subuppersemilattice (\emph{susl}%
) of a hyperarithmetic lattice is isomorphic to an initial segment of $%
\mathcal{D}_{h}$. Now our proof actually makes significant use of the
existence of infima in $\mathcal{L}$ at various points. As it turns out, the
assumption is essential, at least at this level of generality even if we
require $\mathcal{L}$ to be locally finite as well.

\begin{proposition}
\label{uslcode}There is a susl $\mathcal{L}^{\prime }$ of a locally finite
recursive lattice $\mathcal{L}$ which is not isomorphic to any initial
segment of $\mathcal{D}_{h}$.
\end{proposition}

\begin{proof}
The construction is an elaboration of that referred to in Theorem \ref{ctrex}
that exploits the possibility of exact pairs for ideals in usls to make the
initial lattice locally finite. The basic construction of the finitely
generated successor model of Shore [1981] as modified in Shore [2007],
[2008] begins with special elements designated by $%
d_{0},e_{0},e_{1,}f_{0},f_{1},p$ and $q$. They contain a sequence $d_{n}$ of
elements of order type $\omega $ generated by the special elements as
follows:

\begin{enumerate}
\item[$(\ast )$] $(d_{2n}\vee e_{0})\wedge f_{1}=d_{2n+1}$ and

\item[$(\ast \ast )$] $(d_{2n+1}\vee e_{1})\wedge f_{0}=d_{2n+2}$.
\end{enumerate}

\noindent In addition we require that $p\ngeq q$ and $p\vee d_{n}\geq q$ for
each $n$. We then code a set $X$ by adding two additional special elements $%
c_{X}$ and $\bar{c}_{X}$ such that $d_{n}\leq c_{X}$ for $n\in X$, $%
d_{n}\wedge c_{X}=0$ for $n\notin X$, $d_{n}\leq \bar{c}_{X}$ for $n\notin X$
and $d_{n}\wedge \bar{c}_{X}=0$ for $n\in X$. (So, in particular, $\exists
x(0<x\leq d_{n},c_{X})\rightarrow d_{n}\leq c_{X}$ and $\exists x(0<x\leq
d_{n},\bar{c}_{X})\rightarrow d_{n}\leq \bar{c}_{X}$.)

We adjust this procedure to make the lattice locally finite. In place of $%
d_{0}$ we have a set of elements $d_{i,0}$ for $i\in \mathbb{\omega }$. For
each $i$ the sequence generated by the schemes $(\ast )$ and $(\ast \ast )$
now terminates after $i$ steps producing sequences $d_{i,0},\ldots ,d_{i,i}$
of length $i+1$ by having $(d_{i,i}\vee e_{0})\wedge f_{1}=0$ if $i$ is even
and $(d_{i,i}\vee e_{1})\wedge f_{0}=0$ if $i$ is odd. We now require that $%
p\vee d_{i,j}\geq q$ for every $j\leq i\in \mathbb{\omega }$ and add on new
elements $\hat{p}\ngeq \hat{q}$ such that $\hat{p}\vee d_{i,j}\geq \hat{q}$
if and only if $j=i$. In place of $c_{X}$ and $\bar{c}_{X}$ we have one
fixed pair $c$ and $\bar{c}$ that are above all the $d_{i,0}$ and no other
of the previously mentioned elements. We complete this description to a
lattice $\mathcal{L}$ in a way that respects the given ordering and
specified join and meet relations and makes $c\wedge \bar{c}$ a minimal
upper bound of the ideal generated by the $d_{i,0}$ with no other nonzero
elements below it. We give more details after we see what properties are
needed to make our coding of sets in a susl of $\mathcal{L}$ be sufficiently
flexible to show that some such are not isomorphic to initial segments of $%
\mathcal{D}_{h}$.

Given a set $X$ we want to code $X$ into a susl $\mathcal{K}$ of $\mathcal{L}
$ by taking the susl of $\mathcal{L}$ generated by the special elements $c,%
\bar{c},e_{0},e_{1,}f_{0},f_{1},p,q,\hat{p}$ and $\hat{q}$ and the $d_{i,j}$
for $i\in X$. In particular, $c$ and $\bar{c}$ will now be an exact pair for
the ideal generated by the $d_{i,0}$ for $i\in X$. We want to guarantee that 
$n\in X\Leftrightarrow $ there is a sequence $x_{0},\ldots ,x_{n}$ with $%
x_{0}\leq c,\bar{c}$; $x_{2m+1}\leq x_{2m}\vee e_{0},f_{1}$ and $%
x_{2m+2}\leq x_{2m+1}\vee e_{1},f_{0}$ for $2m+1,2m+2\leq n$; $x_{m}\vee
p\geq q$ for $m\leq n$ and $x_{n}\vee \hat{p}\geq \hat{q}$.

We claim that in this case $X$ is $\Pi _{1}^{1}$ in the top $G$ of any
embedding of $\mathcal{K}$ as an initial segment. As existential
quantification over sets hyperarithmetic in $G$ and the relation $A\leq
_{h}B $ for sets given as hyperarithmetic in $G$ are both $\Pi _{1}^{1}$ in $%
G$ and the join operator is recursive (on indices), it is clear that the
specified relation on $n$ is $\Pi _{1}^{1}(G)$ and that it holds of every $%
n\in X$. What remains to verify is that it holds only of $n\in X$. So
suppose there are $x_{0},\ldots ,x_{n}$ as described. For $x_{0}\leq c,\bar{c%
}$, we want $x_{0}$ to be the join of finitely many $d_{i,0}$ by making
these the only elements of $\mathcal{K}$ below both $c$ and $\bar{c}$. Say
for definiteness that $x_{0}=d_{i_{1},0}\vee \cdots \vee d_{i_{k},0}$. We
also arrange our lattice so that $(x_{0}\vee e_{0})\wedge
f_{1}=d_{i_{1},1}\vee \cdots \vee d_{i_{k},1}$ and so $x_{1}\leq
d_{i_{1},1}\vee \cdots \vee d_{i_{k},1}$. In general, we arrange our lattice
so that $(x_{2n}\vee e_{0})\wedge f_{1}=d_{i_{1},2n+1}\vee \cdots \vee
d_{i_{k},2n+1}$ and $(x_{2n+1}\vee e_{1})\wedge f_{0}=d_{i_{1},2n+2}\vee
\cdots \vee d_{i_{k},2n+2}$ where we understand that for $m>i$, $d_{i,m}=0$.
Thus $x_{m}\leq d_{i_{1},m}\vee \cdots \vee d_{i_{k},m}$. The requirements
that $x_{m}\vee p\geq q$ guarantee that $x_{m}>0$ and, by making the $%
d_{i,j} $ minimal elements of the lattice, they must be above some nonzero $%
d_{i_{k},m}$. Finally, we guarantee that the only way such an $x_{m}$ can
join $\hat{p}$ above $\hat{q}$ is for it to be above some $d_{i_{k},i_{k}}$
but these elements are in $\mathcal{K}$ if and only if $i_{k}\in X$ as
required.

Now to be more specific about the structure of $\mathcal{L}$ we specify its
elements and the order on them that will give a lattice with all the desired
properties. We begin, of course, with $0$ and $1$. The elements $%
e_{0},e_{1},p,q,\hat{p},\hat{q}$ and $d_{i,j}$ (for $j\leq i\in \mathbb{%
\omega }$) are minimal nonzero elements of $\mathcal{L}$. We extend the $%
d_{i,j}$ freely to an usl $\mathcal{L}^{\prime }$ by taking all formal
finite joins. This imposes a lattice structure on this set as well since
each of the new elements has only finitely many elements below it. We next
let $\hat{p}$ act on the usl $\mathcal{L}^{\prime }$ as an order isomorphism
under join (so for $x,y\in \mathcal{L}^{\prime }$, $x\vee \hat{p}\leq y\vee 
\hat{p}\Leftrightarrow x\leq y$). Let $\mathcal{L}^{\prime \prime }$ be the
susl of $\mathcal{L}^{\prime }$ generated by the $d_{i,j}$ with $j<i$. For $%
x\in \mathcal{L}^{\prime \prime }$, no elements other than $\hat{p}$ and
those $y\leq x$ are below $x\vee \hat{p}$. For $x\in \mathcal{L}^{\prime }-%
\mathcal{L}^{\prime \prime }$ we also put $x\vee \hat{p}\geq \hat{q}$.
Joining $e_{0}$ with members of $\mathcal{L}^{\prime }$ also acts as an
order isomorphism except that, for any $m\in \mathbb{\omega }$ and $\langle
i_{k}\rangle ,\langle j_{k}\rangle \in \mathbb{\omega }^{m}$ such that $%
~\forall k<m(j_{k}\leq i_{k})$, we make the following identification: 
\begin{equation*}
e_{0}\vee \bigvee \{d_{i_{k},j_{k}}|k<m\}=e_{0}\vee \bigvee
\{d_{i_{k},j_{k}}|k<m\}\vee \bigvee \{d_{i_{k},j_{k}+1}|k\mathbb{<}%
m,j_{k}<i_{k}~\text{and }j_{k}~\text{is even}\}.
\end{equation*}

We do the same for joining with $e_{1}$ except that we change
\textquotedblleft even\textquotedblright\ to \textquotedblleft
odd\textquotedblright . We let $f_{0}$ ($f_{1}$) be above the usl generated
by the $d_{i,j}$ for odd (even) $j\leq i$ and put in a new element $c\wedge 
\bar{c}$ (below both $c$ and $\bar{c}$) which is above the usl generated by
the $d_{i,0}$. Any order relation $x\leq y$ not dictated by these
definitions is taken to be false. So for example, $e_{0}\vee
e_{1}=1=f_{0}\vee f_{1}$, $e_{0}\wedge e_{1}=0=f_{0}\wedge f_{1}$, $x\vee
p=1 $ for any $x\neq 0$, $d_{i,i}\vee \hat{p}=1$ for every $i$, etc. It is
tedious but straightforward to verify that the partial order so defined
imposes on the elements described a lattice structure $\mathcal{L}$ (i.e.\
for every $x,y\in \mathcal{L}$ there is a least upper bound and a greatest
lower bound for the pair in the defined ordering) that has all the desired
properties.

If we now take $X$ to be, for example, the complement of $\mathcal{O}$, then
the top of any embedding of $\mathcal{K}$ as an initial segment would be
above $\mathcal{O}$ and so the degrees below it could not be isomorphic to $%
\mathcal{K}$ (as, for example, every countable partial order can be embedded
in $\mathcal{D}_{h}$ below $\deg _{h}\mathcal{O}$ by Feferman [1965]).
\end{proof}

On the other hand, there are initial segments of $\mathcal{D}_{h}$ which are
not lattices. Indeed, the usual proof that $\mathcal{D}_{T}$ is not a
lattice can be carried out for $\mathcal{D}_{h}$ by using Cohen forcing in
the hyperarithmetic setting to show that the degrees below $\mathcal{O}$ are
not a lattice (as is pointed out in Odifreddi [1983a, Proposition 8.3(b)].
Thus we have our next question.

\begin{question}
Which countable usls are isomorphic to initial segments of $\mathcal{D}_{h}$?
\end{question}

This question seems wide open and, by the above Proposition, must need some
new construction technique. One attractive possibility would be a positive
answer to the following.

\begin{question}
\label{hypusl}Is every hyperarithmetic usl isomorphic to an initial segment
of $\mathcal{D}_{h}$?
\end{question}

Here even the simplest example seems to need some new idea.

\begin{question}
Is the usl consisting precisely of an initial segment $x_{n} $ of type $%
\omega $ with an exact pair $x,y$ and their join (i.e.\ $0=x_{0}$, $\forall
n(x_{n}<x_{n+1}\leq x,y)$ and $x\vee y=1$) isomorphic to an initial segment
of $\mathcal{D}_{h}$?
\end{question}

On the other hand, we have no particularly plausible candidate for a
counterexample to Question \ref{hypusl}.

\end{document}